\documentclass{aptpub}
\usepackage[utf8]{inputenc}
\usepackage[T1]{fontenc}
\usepackage{lmodern}

\usepackage{xcolor}
\usepackage{amsfonts,amsmath,aliascnt,amssymb,bbm,dsfont,mathrsfs}
\usepackage{enumerate,mathtools}
\usepackage[retainorgcmds]{IEEEtrantools}

\usepackage[numbers]{natbib} 

\usepackage{hyperref} 
\hypersetup{final,bookmarksopen=true,hidelinks}
\usepackage[atend]{bookmark} 

\authornames{JEAN BERTOIN AND ALEXANDER R. WATSON} 
\shorttitle{Probabilistic aspects of critical growth-fragmentation equations} 

\let\qedhere\relax

\providecommand{\email}[1]{\href{mailto:#1}{\nolinkurl{#1}}}

\makeatletter
\newcommand{\itemlab}[2]{%
\def\@itemlabel{#1}
\item
\def\@currentlabel{#1}\label{#2}}
\makeatother

\let\epsilon\varepsilon

\def\R{\mathbb{R}}
\def\C{\mathbb{C}}
\def\N{\mathbb{N}}
\def\E{\mathbb{E}}
\def\P{\mathbb{P}}

\def\R{\mathbb{R}}

\def \e{{\mathrm{e}}}
\def \d{{\mathrm{d}}}

\newcommand\Zb{\mathbf{Z}}

\newcommand\Yy{\mathcal{Y}}
\newcommand\dd{\mathrm{d}}

\newcommand\Ind{\mathbbm{1}}
\newcommand\Indic[1]{\Ind_{\{#1\}}}

\let\from\colon

\newcommand\Aa{\mathcal{A}^{(\alpha)}}

\newcommand\skippar{\medskip\noindent}
\newcommand\define[1]{\emph{#1}}
\newcommand\FFt{(\mathscr{F}_t)_{t\ge 0}}
\newcommand\ssP{\mathrm{P}}
\newcommand\LevP{\P}
\newcommand\lP{\P}
\newcommand\ssE{\mathrm{E}}

\newcommand\Xp{X_{_+}}

\newcommand\Cb{\mathcal{C}_b(0,\infty)}
\newcommand\Ctest{\mathcal{C}_c^\infty}

\newcommand\ZZ{\mathbb{Z}}

\newcommand\tree{\mathcal{U}}
\newcommand\rays{\mathcal{V}}
\newcommand\eve{\varnothing}
\newcommand\parent{\prec}
\newcommand\parenteq{\preceq}
\newcommand\Xf{\mathcal{X}}
\newcommand\given{\mathbin{\lvert}}
\newcommand\dint{\displaystyle\int}

\newcommand\Nn{\mathcal{N}}

\DeclarePairedDelimiterX\ip[2]{\langle}{\rangle}{#1,#2}
\DeclarePairedDelimiter\abs{\lvert}{\rvert}

\DeclareMathOperator{\dom}{dom}

\newenvironment{eqnarr*}{\begin{IEEEeqnarray*}{rCl}}{\end{IEEEeqnarray*}\ignorespacesafterend}

\renewcommand{\eqref}[1]{\hyperref[#1]{(\ref*{#1})}}
\def\pref{\eqref} 
\newcommand{\refpref}[2]{\hyperref[#2]{\ref*{#1}(\ref*{#2})}}
\newcommand{\autorefpref}[2]{\hyperref[#2]{\autoref*{#1}(\ref*{#2})}}
\newcommand{\kindref}[2]{\hyperref[#2]{#1~\ref*{#2}}}

\numberwithin{equation}{section}  
\begin{document}

\title{Probabilistic aspects of critical \\ growth-fragmentation equations} 

\authorone[University of Zurich]{Jean Bertoin} 

\addressone{Institute of Mathematics, University of Zurich, Winterthurerstrasse 190, 8057 Z\"urich, Switzerland} 

\authortwo[University of Manchester]{Alexander R. Watson}
\addresstwo{School of Mathematics, University of Manchester, Manchester, M13 9PL, UK}

\begin{abstract}
  The self-similar growth-fragmentation equation describes the evolution
  of a medium in which particles grow and divide as time proceeds, with the
  growth and splitting of each particle depending only upon its size.
  The critical case of the equation, in which the growth and division rates 
  balance one another, was
  considered by Doumic and Escobedo \cite{DouEsc} in the 
  homogeneous case where the rates do not depend on the particle size.
  Here, we study the general self-similar case, using a probabilistic approach
  based on L\'evy processes and positive self-similar Markov processes
  which also permits us to analyse quite general splitting rates.
  Whereas existence and uniqueness of the solution are rather easy to establish 
  in the homogeneous case, the equation in the non-homogeneous case
  has some surprising features.
  In particular, using the fact that certain 
  self-similar Markov processes can enter $(0,\infty)$ continuously from either $0$ 
  or $\infty$, we exhibit unexpected spontaneous generation of mass in the solutions.
\end{abstract}

\keywords{Growth-fragmentation equation;
  self-similarity;
  self-similar Markov process;
  branching process.} 

\ams{35Q92}{45K05; 60G18; 60G51} 

\section{Introduction}

The growth-fragmentation equation is a linear differential equation
intended to describe the evolution of  a medium in which particles grow and split as time passes.
It is frequently expressed in terms of the concentration of particles with size $x>0$ at time $t$,
say $u(t,x)$, as follows:
\begin{equation}\label{EqGFG}
  \partial_tu(t,x)+ \partial_x(\tau(x)u(t,x))+B(x)u(t,x)=\int_x^{\infty}k(y,x)B(y) u(t,y) \d y,
\end{equation}
where $\tau(x)$ is the speed of growth of a particle with size $x$,
$B(x)$ the rate at which a particle of size $x$ splits, and $k(y,x)=k(x-y,x)$
twice the probability density that a particle with size $x$ splits into two
particles with size $yx$ and $(1-y)x$ (the factor $2$ is due to the symmetry of the splitting events).
This type of equation 
has a variety of applications in mathematical modeling,
notably in biology where particles should be thought of as cells,
and has motivated several works in the recent years;
see, for example, \cite{DouEsc}, which also contains a summary of some
recent literature.

We are interested here in the situation  $\tau(x)=c x^{\alpha+1}$, $B(x)=x^{\alpha}$
for some $\alpha\in\R$  and $k$ has the form $k(y,x)=x^{-1}k_0(y/x)$;
for these parameters, \eqref{EqGFG} possesses a useful self-similarity property.  This is referred to
as the \emph{critical} case by Doumic and Escobedo \cite{DouEsc}, who studied in depth the
situation when, additionally, $\alpha=0$.
For our purposes, it will be more convenient to write the equation in
weak form, as follows.
For $x>0$ and $y\in[1/2,1)$, we write $(y\mid x)$ for the pair $\{yx,(1-y)x\}$, which we view 
as the dislocation of a mass $x$ into two smaller masses, and then for every function $f\from (0,\infty)\to \R$,  we set
$$f(y\mid x) \coloneqq f(yx)+f((1-y)x).$$
Consider test functions $f\in{\mathcal C}^{\infty}_c(0,\infty)$,
that is, $f$ is infinitely differentiable with compact support. For a measure
$\mu$ on $(0,\infty)$, we write $\ip{\mu}{f} \coloneqq \int_{(0,\infty)} f(x) \, \mu(\dd x)$.
By integrating \eqref{EqGFG}, we obtain the equation
\begin{equation}\label{EqFG}
  \partial_t \langle \mu_t, f\rangle
  = \left\langle \mu_t, {\mathcal L}f\right\rangle, 
\end{equation}
where $\mu_t(\d x) = u(t,x)\d x$ and the operator ${\mathcal L}$ has the form 
\begin{equation}\label{Eqop0}
  {\mathcal L}f(x) = x^{\alpha}\left( cxf'(x) + \int_{[1/2,1)}\left(f(y\mid x)-f(x)\right) K(\d y)\right) ,
\end{equation}
where
$$K(\d y) \coloneqq k_0(y)\d y = k_0(1-y)\d y\,,\qquad  y\in[1/2,1),$$
is referred to as the \emph{dislocation measure}.
The advantage of this formulation is that we do not require
absolute continuity of the solution $\mu_t$ or the dislocation measure $K$. 
More generally, one might also consider non-binary dislocation measures,
but we refrain from doing so in this work in order to simplify the presentation.

\skippar
In this article, we study the equation \eqref{EqFG} for operators of the form 
\begin{equation} \label{Eqop1}
  {\mathcal L}_{\alpha}f(x) 
  \coloneqq 
  x^{\alpha}\biggl( a x^2 f''(x)+ b xf'(x)
  + \int_{[1/2,1)}\bigl(f(y\mid x)-f(x)+xf'(x)(1-y)\bigr) K(\d y)\biggr) ,
\end{equation}
where $a\geq 0$, $b\in\R$, and we now only assume that the measure $K$ satisfies the weaker requirement
\begin{equation}\label{Eqcondnu}
  \int_{[1/2,1)}(1-y)^2 K(\d y)<\infty.
\end{equation}
Our notion of a \emph{solution} of \eqref{EqFG} is a collection of locally finite 
measures $(\mu_t)_{t\ge 0}$ on $(0,\infty)$ such that,
for every $f \in \Ctest(0,\infty)$ and $t\geq 0$, there is the identity 
$$ \langle \mu_t, f\rangle =  \langle \mu_0, f\rangle + \int_0^t   \left\langle \mu_s, {\mathcal L}f\right\rangle \d s .$$
(This requires implicitly that 
$s\mapsto \left\langle \mu_s, {\mathcal L}f\right \rangle$ be a well-defined, locally integrable function, 
and  in particular the family $(\mu_t)_{t\ge 0}$ is then vaguely continuous).

We offer a comparison between the original operator \eqref{Eqop0} and our operator \eqref{Eqop1}.
Besides the appearance of a second order derivative, 
there is a new term $xf'(x)(1-y)$ in the integral in \eqref{Eqop1}.
The latter should be interpreted as an additional growth term which, in some sense, 
balances the accumulation of small dislocations. 
We stress that \eqref{Eqcondnu} is the necessary and sufficient condition for \eqref{Eqop1} to be
well-defined, and  that when the measure $K$ is finite (or at least fulfills $\int_{[1/2,1)}(1-y) K(\d y)<\infty$),
every operator of the form \eqref{Eqop0} can also be expressed in the form
\eqref{Eqop1}. Our motivation for considering this more general setting stems from the recent 
work \cite{BeCF}, in which a new class of growth-fragmentation stochastic processes  is constructed
such that, loosely speaking, the strong rates of dislocation that would instantaneously shatter the
entire mass can be somehow compensated by an intense growth; the dislocation measure associated
with such a fragmentation process need only satisfy \eqref{Eqcondnu}.

In short, the purpose of this work is to demonstrate the usefulness of
some probabilistic methods for the study of these critical growth-fragmentation equations.
More precisely, we shall see that solutions to 
\eqref{EqFG} 
for ${\mathcal L}={\mathcal L}_{\alpha}$ can be related to 
the one-dimensional distributions of certain self-similar Markov processes,
and this will enable us to reveal some rather unexpected features of the former. 
Although in the homogeneous case $\alpha=0$, 
we establish existence and uniqueness of the solution in full generality, 
this feature is lost for $\alpha\neq 0$.
In particular, we shall see that 
under a fairly general assumption on the parameters of the model, 
the critical growth-fragmentation equation permits spontaneous generation,
i.e. there exist non-degenerate solutions starting from the null initial condition.

We need some notation before describing more precisely our main results. 
We first introduce the function $\kappa \from [0,\infty)\to (-\infty,\infty]$ which plays a major role in our approach and is given by:
\begin{equation}\label{eqkappa}
  \kappa(q)
  \coloneqq 
  a q^2 + (b-a) q + \int_{[1/2,1)}\left(y^q +(1-y)^q- 1 +q(1-y)\right) K(\d y) , \qquad q\geq 0.
\end{equation}
There are two principal `Malthusian hypotheses' which we will require when $\alpha\neq 0$:
\begin{enumerate}
  \itemlab{$(M_+)$}{i:M1}
    $\inf_{q\geq 0} \kappa(q) <0.$
  \itemlab{$(M_-)$}{i:M2}
    There exist $0\le \omega_{_-}< \omega_+$ and $\epsilon>0$ such that
    $\kappa(\omega_{_-})=\kappa(\omega_+)=0$ and
    \linebreak 
    $\kappa(\omega_--\epsilon)<\infty$.
\end{enumerate}

We now summarise our main results, 
deferring their proofs to the body of the article.

\begin{itemize}
  \item
    For $\alpha = 0$, the equation \eqref{EqFG} with operator \eqref{Eqop1}
    and initial condition $\mu_0 = \delta_1$ has a unique solution.
  
  \item
    For $\alpha < 0$: suppose that \ref{i:M1} holds. Then,
    the equation \eqref{EqFG} with operator \eqref{Eqop1}
    has a solution with initial condition $\mu_0 = \delta_1$.
    There exists further a non-degenerate solution started from $\mu_0=0$;
    in particular, uniqueness fails.

  \item
    For $\alpha > 0$: if \ref{i:M1} holds, then
    the equation \eqref{EqFG} with operator \eqref{Eqop1}
    has a solution with initial condition $\mu_0 = \delta_1$.
    If \ref{i:M2} holds,
    then there also exists a non-degenerate solution started from $\mu_0=0$;
    again, uniqueness fails.
\end{itemize}

We shall also observe that under essentially the converse assumption to \ref{i:M1},
namely that $\inf_{q\ge0} \kappa(q) >0$, 
the particle system that corresponds to the stochastic version of the model
may explode in finite time almost surely.
This is a strong indication that \eqref{EqFG} should have no global solution in the latter case.

\skippar
The rest of this article is organized as follows. 
The next section provides brief preliminaries on the function $\kappa$ and the use of the Mellin transform
in the study of growth-fragmentation equations.
Section \ref{s:hom} is devoted to the homogeneous case $\alpha=0$, and then the general self-similar case
$\alpha \neq 0$ is presented in section \ref{s:ss}.
In section \ref{s:explosion}, 
we investigate  a stochastic model related to the growth-fragmentation equation,
to demonstrate that explosion may occur when the Malthusian hypothesis fails.
Finally, in section \ref{s:bps},
we briefly discuss another interpretation of the growth-fragmentation equation 
in terms of branching particle systems and many-to-one formulas,
placing the results of sections \ref{s:hom} and \ref{s:ss} in context.

\section{The  Mellin transform and the growth-fragmentation equation}
\label{s:prelim}

We observe first that, for any $\alpha \in \R$, the operator 
${\mathcal L}_{\alpha}$ 
fulfills a self-similarity property.
Specifically,  
for every $c>0$, 
if we denote by $\varphi_c(x)=cx$ the dilation function with factor $c$,
then for a generic $f\in{\mathcal C}^{\infty}_c(0,\infty)$, 
there is the identity
\begin{equation}\label{EqSSL}  
  {\mathcal L}_{\alpha}(f\circ \varphi_c)=c^{-\alpha}\left({\mathcal L}_{\alpha}f\right)\circ \varphi_c.
\end{equation}
As a consequence, if $(\mu_t)_{t\geq 0}$ is a solution to \eqref{EqFG} for all
$f\in{\mathcal C}^{\infty}_c(0,\infty)$
with initial condition $\mu_0=\delta_1$,
and if $\tilde \mu_t$ denotes the image of $\mu_{t}$ by the dilation $\varphi_c$,
then $(\tilde \mu_{c^{\alpha}t})_{t\geq 0}$ is a solution to \eqref{EqFG} 
for all $f\in{\mathcal C}^{\infty}_c(0,\infty)$  with initial condition $\tilde\mu_0=\delta_c$.
For the sake of simplicity,
we shall therefore focus on the growth-fragmentation equation
with initial condition $\mu_0=\delta_1$,
since this does not induce any loss of generality.

Recall that the function $\kappa$ has been introduced in \eqref{eqkappa};
its domain  is clarified by the following result.
 
\begin{lem} \mbox{} \label{Lbounds}
  \begin{enumerate}[(i)]
    \item\relax\label{i:Lbounds:def}
       For every $q\geq 0$, $\kappa(q)$ is well-defined with values in $(-\infty,\infty]$.
       The function $\kappa$ is convex, and
       we define $\dom \kappa = \{ q \geq 0 : \kappa(q)<\infty\}$.
    \item\relax\label{i:Lbounds:dom}
      For $q\geq 0$, $\kappa(q)<\infty$ if and only if $\int_{[1/2,1)}(1-y)^qK(\d y)<\infty$,
      and in particular $[2,\infty)\subseteq \dom \kappa$.
    \item\relax\label{i:Lbounds:asym}
      For  every function $f$ in ${\mathcal C}^{\infty}_c(0,\infty)$, ${\mathcal L}_{\alpha}f$
      is a continuous function on $(0,\infty)$ and is identically $0$ in some neighborhood of $0$.
      Furthermore, ${\mathcal L}_{\alpha}f(x)=o(x^{q+\alpha})$ as $x\to \infty$ for every
      $q\in \dom \kappa$, and thus in particular for $q=2$.
 \end{enumerate}
\end{lem}
\begin{proof}
  \begin{itemize}
    \item[(i--ii)]
      First, the integral $\int_{[1/2,1)}\left(y^q - 1 +q(1-y)\right) K(\d y)$ converges
      absolutely thanks to \eqref{Eqcondnu}, since $y^q - 1 +q(1-y)=O((1-y)^2)$. 
      It follows that $\kappa(q)$ is then well-defined with values in $(-\infty,\infty)$
      if and only if $\int_{[1/2,1)}(1-y)^qK(\d y)<\infty$, and otherwise $\kappa(q)=\infty$. 
    \item[(iii)]
      The first assertions are straightforward, and so we check only the last one.
      Take $q\in \dom \kappa$ and recall from above that $\int_{[1/2,1)}(1-y)^qK(\d y)<\infty$.
      This entails $K([1/2,1-\varepsilon))=o(\varepsilon^{-q})$ as $\varepsilon\to 0+$.
      Since $f$ has compact support in $(0,\infty)$, we have for $x$ sufficiently large that
      $\mathcal L_0f(x)=\int_{[1/2,1)}f(x(1-y))K(\d y)$, and we easily conclude that ${\mathcal L}_0f(x)=o(x^q)$. 
      \qedhere
  \end{itemize}
\end{proof}

Doumic and Escobedo \cite{DouEsc} studied certain growth-fragmentation equations
with homogeneous operators given by \eqref{Eqop0} for $\alpha=0$, and
observed that the Mellin transform plays an important role.
In this direction, it is useful to introduce the notation
$h_q\from (0,\infty)\to (0,\infty)$, $h_q(x)=x^q$, for the power function with exponent $q$, and 
recall that the Mellin transform of a measurable function $f\colon(0,\infty)\to \R$ is defined for $z\in\C$ by
$${\mathcal M}f(z)\coloneqq  \int_0^{\infty} f(x) x^{z-1}\d x $$
whenever the integral in the right-hand side converges. 
It follows from \autoref{Lbounds} that the Mellin transform of ${\mathcal L}_{\alpha}f$ 
is well defined for all $z < -2-\alpha$, or more generally for all $z$ such that $-z-\alpha \in (\dom \kappa)^\circ$.

The role of  $\kappa$ in this study stems from the following lemma,
which is easily checked from elementary properties of the Mellin transform;
see \cite[\S 1.1]{DouEsc} and \cite[\S 12.3]{Davies}.
\begin{lem}\mbox{} \label{L0}  
  \begin{enumerate}[(i)]
    \item\relax\label{i:L0:eigenfunction}
      Let $q \in \dom \kappa$. Then,
      $$ {\mathcal L}_{\alpha}h_q(x)= \kappa(q) h_{q+\alpha}(x), \qquad x > 0.$$
      In particular, for $\alpha=0$,  $h_q$ is an eigenfunction for ${\mathcal L}_{0}$ with eigenvalue $\kappa(q)$. 
    \item For every $q$ such that $q-\alpha \in (\dom \kappa)^\circ$
      and every $f\in{\mathcal C}^{\infty}_c(0,\infty)$, there is the identity
    $${\mathcal M}({\mathcal L}_{\alpha}f)(-q)= \kappa(q-\alpha) {\mathcal M}f(-q+\alpha).$$
  \end{enumerate}
\end{lem}

\section{The homogeneous case}
\label{s:hom}

Throughout this section, we assume that $\alpha=0$, and refer to this case as \textit{homogeneous}.
Recall that when $a=0$ and the dislocation measure $K$ fulfills the stronger condition 
\begin{equation}\label{Eqcondnu1}
  \int_{[1/2,1)}(1-y)K(\d y)<\infty, 
\end{equation}
 then we can express the operator ${\mathcal L}_{0}$ in the simpler form
\begin{equation}\label{e:L-c-nu}
  {\mathcal L}_{c,K}f(x) \coloneqq cxf'(x) + \int_{[1/2,1)}\left(f(y\mid x)-f(x)\right) K(\d y)
\end{equation}
with $c=b+\int_{[1/2,1)}(1-y) K(\d y)$. 
This situation was considered in depth by \citet{DouEsc}, and most of the results of this section
should be viewed as extensions of those in \cite{DouEsc} to the case when either $a>0$, or $K$
fulfills \eqref{Eqcondnu} but not  \eqref{Eqcondnu1}. Furthermore, the case
$c \le 0$ was considered by \citet{Haas} using the same method that we employ below.

\subsection{Main results}
\label{s:zero-main}

The key observation in \autorefpref{L0}{i:L0:eigenfunction} that
power functions $h_q$ are eigenfunctions of the operator
${\mathcal L}_{0}$, underlies the analysis of the homogeneous case.
Specifically,  if we knew that $(\mu_t)_{t\geq 0}$ solves \eqref{EqFG} with $f=h_q$ for $q\geq 2$,
then the Mellin transform of $\mu_t$, 
${M}_t(z)=\langle \mu_t, h_{z-1}\rangle$,  would solve the linear equation
\begin{equation}\label{EqMelh}
  \partial {M}_t(q+1) = \kappa(q){M}_t(q+1).
\end{equation}
Focussing for simplicity on the initial condition $\mu_0=\delta_1$,  so that ${M}_0(q)=1$, we would find 
\begin{equation}\label{SolMelh}
  {M}_t(q+1)=\exp(t\kappa(q)).
\end{equation}

In order to check that \eqref{SolMelh} is indeed the Mellin transform of a positive measure, we define,
for every  $\omega \in \dom \kappa$,
a new function by shifting $\kappa$:
\[
  \Phi_{\omega}(q)\coloneqq \kappa(\omega+q)-\kappa(\omega),\qquad q\geq 0.
\]
This is a smooth, convex function, and it has a simple probabilistic interpretation,
which will play a major role throughout. In this direction, recall first that 
a \define{L\'evy process} is a stochastic process 
issued from the origin
with stationary and independent increments and 
c\`adl\`ag paths. It is further called \define{spectrally negative} if all its jumps are negative. 
If $\xi: = (\xi(t))_{t\geq 0}$ is a spectrally negative L\'evy process
with law $\lP$, then for all $t\geq 0$ and $\theta\in \R$,
the \emph{Laplace exponent} $\Phi$, given by 
\[ 
  \E\bigl[ \exp(q \xi(1)) \bigr] = \exp(\Phi(q)) ,
\]
is well defined (and finite) for all $q\geq 0$, and 
satisfies the classical L\'evy-Khintchin formula
\[
  \Phi(q) =  \mathtt{a}q + \frac{1}{2}\sigma^2q^2 + 
  \int_{(-\infty, 0)} ( \e^{qx} -1 + q x\Indic{|x|\leq 1})\Upsilon(\dd x) , \qquad q \ge 0,
\]
where $\mathtt{a}\in\mathbb{R}$, $\sigma\geq 0$, and $\Upsilon$ is a measure
(the \define{L\'evy measure}) on
$(-\infty, 0)$ such that 
$$\int_{(-\infty, 0)}(1\wedge x^2)\Upsilon(\dd x)<\infty.$$
is in fact the Laplace exponent of a spectrally negative Lévy process;
see \cite[Theorem 8.1]{Sato}.

\begin{lem}\relax \label{L1}
  Let $\omega \in \dom\kappa$. Then:
  \begin{enumerate}[(i)]
    \item 
      the function $\Phi_\omega$ is the Laplace
      exponent of a spectrally negative Lévy process, which we will call\
      $\xi_\omega = (\xi_\omega(t))_{t\ge 0}$.
    \item
      for every $t\geq 0$, there exists a unique probability measure 
      $\rho^{[\omega]}_t$ on $(0,\infty)$ with Mellin transform given by 
      \begin{equation}\label{Eqrho}
        \mathcal{M} \rho^{[\omega]}_t (q+1)
        \coloneqq
        \int_{(0,\infty)} x^q \rho^{[\omega]}_t(\d x) = \exp(t\Phi_\omega(q))\,, \qquad q\geq 0.
      \end{equation}
      The family of measures has the representation $\rho^{[\omega]}_t = \P(\exp(\xi_\omega(t)) \in \cdot)$, for $t\ge 0$.
  \end{enumerate}
\end{lem}
\begin{proof}
  We prove both parts simultaneously, and focus first on the case $\omega=2$, where we write $\Phi \coloneqq \Phi_2$.
  We can express $\Phi$ in the form
  \begin{equation}\label{EqLKPhi}
    \Phi(q)= a q^2 +b'q + \int_{[1/2,1)}(y^q -1+q(1-y))y^2K(\d y) + \int_{[1/2,1)}((1-y)^q -1)(1-y)^2K(\d y)
  \end{equation}
  with
  $$
    b'=3a+b+\int_{[1/2,1)}(1-y)(1-y^2)K(\d y).
  $$

  Let us denote by $\Lambda(\d x)$ the image of $y^2K(\d y)$ by the map $y\mapsto x=\ln y$,
  and $\Pi(\d x)$ the image of $(1-y)^2K(\d y)$ by the map $y\mapsto x=\ln (1-y)$.
  Then, thanks to \eqref{Eqcondnu},  $\Lambda$ is a measure on $[-\ln 2,0)$ with $\int x^2\Lambda(\d x)<\infty$,
  and $\Pi$ is a finite measure on $(-\infty,-\ln 2]$, and there are the identities
  $$ \int_{[1/2,1)}(y^q -1+q(1-y))y^2K(\d y)= \int_{[-\ln 2,0)}(\e^{qx}-1+q(1-\e^x))\Lambda(\d x)$$
  and 
  $$\int_{[1/2,1)}((1-y)^q -1)(1-y)^2K(\d y)=\int_{(-\infty, -\ln 2]}(\e^{qx}-1)\Pi(\d x).$$

  This shows that $\Phi$ is given by  a L\'evy-Khintchin formula, and therefore, $\Phi$ 
  can be viewed as the Laplace exponent of a spectrally negative L\'evy process $\xi=(\xi(t))_{t\geq 0}$ 
  (see Chapter VI in \cite{BeLP} for background), i.e.,
  $$\E\left( \exp(q\xi(t))\right) = \exp(t\Phi(q))\,,\qquad t,  q\geq 0.$$
  We conclude that \eqref{Eqrho} does indeed determine a probability measure 
  $\rho_t$ which arises as the distribution of $\exp(\xi(t))$.

  Finally, if $\omega \ne 2$, we observe that the function $\Phi_\omega$ can
  be written
  \linebreak 
  $\Phi_\omega(q) = \Phi(q+\omega-2)-\Phi(\omega-2)$, which implies that
  $\Phi_\omega$ is given by an Esscher transform of $\Phi$, and hence is also
  the Laplace exponent of a spectrally negative Lévy process;
  see \cite[Theorem 3.9]{Kypr2} or \cite[Theorem 33.1]{Sato}.
\end{proof}

\begin{rem}
  More generally, if $\zeta$ denotes a random time having an exponential distribution, 
  say with parameter $\mathtt{k}\geq 0$, which is further independent of $\xi$,
  then the process 
  $$
    \xi_{\dagger}(t)=\left\{ \begin{matrix} \xi(t) &\hbox{ if }& t<\zeta\\
    -\infty &\hbox{ if }& t\geq \zeta  \end{matrix}\right.
  $$
  is referred to as a killed L\'evy process. Note that if we set $\Phi_{\dagger}(q))= \mathtt{k} + \Phi(q)$, then
  \[ \E\bigl[ \exp(q \xi_{\dagger} (t)) \bigr] = \exp(t\Phi_{\dagger}(q)) ,\]
  with the convention that $\exp(q \xi_{\dagger} (t)) =0$ for $t\geq \zeta$. 
  So \autoref{L1} shows that whenever $\kappa(\omega)\leq 0$,  the function $q \mapsto \kappa(\omega+q)$
  can be viewed as the Laplace exponent of a spectrally negative L\'evy process killed at an
  independent exponential time with parameter $-\kappa(\omega)$. 
\end{rem}

Recall from \autorefpref{Lbounds}{i:Lbounds:dom} that $2 \in \dom\kappa$ always;
we will write $\rho_t$ for $\rho^{[2]}_t$. Since $\rho_t$
is guaranteed to exist, this collection of measures will play a particular
role in the case $\alpha=0$. We stress that in the cases $\alpha < 0$ and $\alpha >0$, 
we will need to choose different values of $\omega$,
and the notation $\rho_t$ will then refer to a different distribution.

We point out the following property of the probability measures $\rho^{[\omega]}_t$,
which essentially rephrases Kolmogorov's forward equation.  

\begin{cor}\relax\label{C0}
  The family of probability measures $(\rho^{[\omega]}_t)_{t\geq 0}$ defined in 
  \autoref{L1} depends continuously on the parameter $t$ for the topology of weak convergence.
  Further, for every $g\in {\mathcal C}^{\infty}_c(0,\infty)$, the function 
  $t\mapsto \langle \rho^{[\omega]}_t, g\rangle$ is differentiable with derivative 
  $\partial_t\langle \rho^{[\omega]}_t, g\rangle= \langle \rho^{[\omega]}_t, {\mathcal A}g\rangle$,
  where 
  $${\mathcal A}g(x) \coloneqq x^{-\omega} {\mathcal L}_0(h_\omega g)(x)-\kappa(\omega)g(x)\,, \qquad x>0.$$
\end{cor}
\begin{proof}
  Recall that every L\'evy processes fulfills the Feller property and in particular, for every function 
  $\varphi\in{\mathcal C}_0(\R)$, the map $t\mapsto \E(\varphi(\xi_\omega(t)))$ is continuous. 
  Taking
  \linebreak 
  $\varphi(x)=g(\e^x)$ with $g\in{\mathcal C}_0(0,\infty)$ yields 
  the weak continuity of the map $t\mapsto \rho^{[\omega]}_t$.

  Further, it is well-known that the domain of the infinitesimal generator of a L\'evy process contains
  ${\mathcal C}^{\infty}_c(\R)$ (see, e.g., Theorem 31.5 in \citet{Sato}), and it follows similarly
  that for $g\in {\mathcal C}^{\infty}_c(0,\infty)$, the map
  $t\mapsto \langle \rho^{[\omega]}_t, g\rangle$ is differentiable. 
  To compute the derivative, that is to find the infinitesimal generator, take $q\geq 0$ and recall that
  $h_q(x)=x^q$ for $x>0$. Then simply observe from \eqref{Eqrho} that
  $$\partial_t\langle \rho_t, h_q\rangle
    = \Phi_\omega(q) \exp(t\Phi_\omega(q))
    = \langle \rho^{[\omega]}_t, \Phi_\omega(q)h_q\rangle
    =\langle \rho^{[\omega]}_t, \kappa(q+\omega)h_q-\kappa(\omega)h_q\rangle.$$
  Using \autorefpref{L0}{i:L0:eigenfunction}, we can express 
  $$\kappa(q+\omega)h_q=h_{-\omega}\kappa(q+\omega)h_{q+\omega}=h_{-\omega}{\mathcal L}_0(h_\omega h_q),$$
  which shows that 
  $\partial_t\langle \rho^{[\omega]}_t, h_q\rangle= \langle \rho^{[\omega]}_t, {\mathcal A}h_q\rangle$ 
  for all $q\geq 0$.
  That the same holds when $h_q$ is replaced by a function $g\in{\mathcal C}^{\infty}_c$ now follows from standard arguments, using linear combinations of $h_q$. 
\end{proof}

We would now like to invoke \autoref{L1} to invert the Mellin transform \eqref{SolMelh}, observing that
(using \eqref{Eqrho} with $\rho = \rho^{[2]}$)
$$\exp(t\kappa(q))= \exp(t\kappa(2))\langle \rho_t, x^{q-2}\rangle,$$
 and conclude that 
$$\mu_t(\d x) = \e^{t\kappa(2)} x^{-2}\rho_t(\d x).$$ 
However $h_q\not \in {\mathcal C}^{\infty}_c(0,\infty)$ and we cannot directly apply this simple argument. 
Nonetheless we claim the following. 

\begin{thm}\relax\label{T1}
  The equation \eqref{EqFG}, for $f\in{\mathcal C}^{\infty}_c(0,\infty)$
  and with ${\mathcal L}={\mathcal L}_{0}$, has a unique solution started from 
  $\mu_0 = \delta_1$, given by
  \[ \mu_t(\d x) = \e^{t\kappa(2)} x^{-2}\rho_t(\d x), \qquad t \ge 0, \]
  where $\rho_t$ is the probability measure on $(0,\infty)$ defined by \eqref{Eqrho} for $\omega=2$. 
\end{thm}

\begin{rem}\relax \label{r:LE}
  In particular, the unique solution in \autoref{T1} fulfills
  \linebreak 
  $\langle \mu_t,h_q\rangle = \exp(t\kappa(q))$, as we expected from \eqref{SolMelh}.
  From a probabilistic perspective, this does not come as a surprise. In \cite{BeCF},
  a homogeneous growth-fragmentation stochastic process
  $\Zb(t)=(Z_1(t), Z_2(t), \dotsc)$ was constructed whose evolution
  is, informally speaking, governed by the stochastic growth-fragmentation dynamics
  described in the introduction. Using a spine technique, it may be shown
  (we omit the proof)
  that the solution $(\mu_t)_{t\ge 0}$ has the representation $\ip{\mu_t}{f} = \E \bigl[\sum_{i=1}^\infty f(Z_i(t)) \bigr]$,
  for any $f$ for which the right-hand side is finite; and
  in \cite[Theorem 1]{BeCF}, it is proved that
  $\E\bigl[\sum_{i=1}^{\infty} Z^q_i(t)\bigr]= \exp(t\kappa(q))$ for all $q\geq 2$. 
  We offer a more detailed discussion of the spine technique in \autoref{s:bps}.
\end{rem}

\begin{proof}[Proof of \autoref*{T1}]
  It is straightforward to check that
  $\mu_t(\d x) = \e^{t\kappa(2)} x^{-2}\rho_t(\d x)$ is indeed a solution.
  Specifically, we deduce from  \eqref{Eqrho}, that
  $$\langle \mu_t, h_q\rangle = \exp(t\kappa(2)) \exp(t\Phi(q-2)) = \exp(t\kappa(q)).$$
  We thus see that $(\mu_t)_{t\geq 0}$ solves \eqref{EqFG} with ${\mathcal L}= {\mathcal L}_{0}$
  and $f=h_q$ for every $q\geq 0$,
  and it follows from classical properties of the Mellin transform that this entails that 
  $(\mu_t)_{t\geq 0}$ solves \eqref{EqFG} more generally for all $f\in{\mathcal C}^{\infty}_c(0,\infty)$. 

  Conversely, given a solution $(\mu_t)_{t\geq 0}$ to \eqref{EqFG} with $\mu_0=\delta_1$,
  set 
  \[ \tilde \rho_t(\d x) = \e^{-t\kappa(2)} x^{2}\mu_t(\d x) . \]
  Take $g\in{\mathcal C}^{\infty}_c(0,\infty)$ and define $f(x)=x^2g(x)$ for $x>0$,
  so $f\in{\mathcal C}^{\infty}_c(0,\infty)$. Then we have 
  $\langle \tilde \rho_t, g\rangle = \e^{-t\kappa(2)}\langle \mu_t, f\rangle$ and
  $$
    \partial_t \langle \tilde \rho_t, g\rangle 
    = -\kappa(2)\langle \tilde \rho_t, g\rangle + \e^{-t\kappa(2)}\langle \mu_t, {\mathcal L}_{0}f\rangle,
  $$
  that is,
  \begin{equation} \label{Eqgenexp}
    \partial_t \langle \tilde \rho_t, g\rangle = \langle \tilde \rho_t, {\mathcal A}g\rangle,
  \end{equation}
  with 
  $${\mathcal A}g(x)= x^{-2} {\mathcal L}_{0}f (x)- \kappa(2)g(x),$$
  as in the notation of \autoref{C0}. We can thus
  interpret \eqref{Eqgenexp} as Kolmogorov's forward equation for the infinitesimal generator
  of the Feller process $(\exp(\xi(t)))_{t\geq 0}$.
  This will in turn enable us to identify 
  $\tilde \rho_t= \rho_t$.

  To be precise,  examining the proof of \cite[Proposition 4.9.18]{EK-mp},
  we see that \eqref{Eqgenexp} for 
  all $g\in {\mathcal C}^{\infty}_c(0,\infty)$ has at most one solution (in the sense of a 
  vaguely right-continuous collection
  of measures $(\tilde \rho_t)_{t\ge 0}$)
  so long as the image of ${\mathcal C}^{\infty}_c(0,\infty)$ by
  $\lambda-{\mathcal A}$ is separating (see \cite[p.~112]{EK-mp}) for each $\lambda >0$.
  Since ${\mathcal A}$ is the generator of a Feller process and
  ${\mathcal C}^{\infty}_c(0,\infty)$ is a core (cf. Theorem 31.5 in Sato \cite{Sato}),
  we know that the image of ${\mathcal C}^{\infty}_c(0,\infty)$ by
  $\lambda-{\mathcal A}$ is a dense subset of $\mathcal C_0$,
  and this implies that it is separating.
  If $(\tilde \rho_t)_{t \ge 0}$ is a collection of measures solving \eqref{Eqgenexp},
  then for any $g \in \mathcal C_c^\infty$, the function $t \mapsto \langle \tilde \rho_t,g\rangle$ is right-continuous.
  Hence, the solution of
  \eqref{Eqgenexp} restricted to $\mathcal C_c^\infty$ is unique, and this transfers to
  \eqref{EqFG}.
\end{proof}

\subsection{Some properties of solutions}
\label{s:zero-properties}

We next present some properties of the solution identified in \autoref{T1},
by means of translating known results on Lévy processes.

We first point out that, depending on whether  \eqref{Eqcondnu1} holds and $a=0$, 
the support of the solution $\mu_t$ is bounded or not. 
Specifically, if $a=0$ and \eqref{Eqcondnu1} holds, 
we set 
$$d \coloneqq b + \int_{[1/2,1)}(1-y) K(\d y),$$
and otherwise $d=\infty$. It is easy to verify that $d=\lim_{q\to \infty} q^{-1}\kappa(q)$.

\begin{cor}\label{C1} 
  If $a=0$ and \eqref{Eqcondnu1} holds,
  then for every $t>0$, $\e^{dt}$ is the supremum of the support of $\mu_t$,
  i.e., we have for every $\varepsilon >0$,
  $$\mu_t((\e^{t d}, \infty))=0 \ \hbox{and}\  \mu_t((\e^{t d}-\varepsilon,\e^{t d}])>0.$$ 
\end{cor}
\begin{proof}
  The spectrally 
  negative L\'evy process $\xi = \xi_2$, arising in \autoref{L1},
  has bounded variation with drift coefficient $d$ exactly when the conditions of the result hold,
  and it is then well-known that 
  $td$ is the maximum of the support of the distribution of $\xi(t)$.
  Therefore we have $\rho_t((\e^{t d}, \infty))=0$ and $\rho_t((\e^{t d}-\varepsilon,\e^{t d}])>0$, 
  and our claim follows from \autoref{T1}.
\end{proof}

In the case when the assumptions of \autoref{C1} are not fulfilled, we have the following 
large deviations estimates for the tail 
$ \bar \mu_t(x) \coloneqq \mu_t((x,\infty))$ of $\mu_t$.
Recall that $\kappa$ is a convex function,  and observe that
$\lim_{q \to +\infty} \kappa'(q)=+\infty$ when either $a>0$ or \eqref{Eqcondnu1} fails.
Thus for every $r$ sufficiently large, the equation $\kappa'(q)=r$ has a unique solution
which we denote by $\theta(r)$, and the Legendre-Fenchel transform of $\kappa$ is given by
$$\kappa^*(r)\coloneqq\sup_{q>0}\{rq-\kappa(q)\}=r\theta(r)-\kappa(\theta(r)).$$

\begin{cor}\label{C2} 
  Suppose that $a>0$ or \eqref{Eqcondnu1} fails. 
  Then for every $r>0$ sufficiently large, we have
  $$\lim_{t\to \infty} t^{-1}\ln \bar \mu_t(\e^{tr}) = -\kappa^*(r).$$
\end{cor}
\begin{proof}
  This follows easily from the identity $\langle \mu_t, h_q\rangle = \exp(t\kappa(q))$
  by adapting the classical arguments of Cramér and Chernoff; see, for instance, Theorem 1 in \citet{Bigg}.
\end{proof}

The estimate of \autoref{C2} can easily be reinforced by using the local central limit theorem.
Here is a typical example (compare with Theorem 1.3 in \cite{DouEsc}).

\begin{cor}\label{C3}
  Suppose that $a>0$ or \eqref{Eqcondnu1} fails,
  and further that $\kappa'(q)<0$ for some $q$. 
  Then $\theta(0)$  is well-defined, $0<\kappa''(-\theta(0))<\infty$, 
  and  for every $f\in{\mathcal C}_c$, we have
  $$\langle \mu_t, f\rangle
    \sim
    \frac{\e^{t\kappa(\theta(0))}}{\sqrt{2\pi t \kappa''(\theta(0))}}\int_0^{\infty} f(x) x^{\theta(0)-1} \d x.$$
\end{cor}
\begin{proof}
  The first assertion about the existence of $\theta(0)$ and $\kappa''(\theta(0))$ are immediate from the convexity 
  of $\kappa$ and the fact that $\lim_{+\infty} \kappa=+\infty$.

  The function $\tilde \Phi(q)\coloneqq\kappa(q+\theta(0))-\kappa(\theta(0)) = \Phi(q+\theta(0))-\Phi(\theta(0))$ 
  is the Laplace exponent of a spectrally negative L\'evy process $(\tilde \xi(t))_{t\geq 0}$ which is centered and
  has finite variance $\kappa''(\theta(0))$. 
  Further, we see from Esscher transform and \autoref{T1} that
  $$\mu_t(\d x) = \e^{t\kappa(\theta(0))} x^{-\theta(0)} \P(\exp(\tilde \xi(t))\in \d x).$$
  Our claim then follows readily from the local central limit theorem for the L\'evy process.
\end{proof}

\begin{cor}\label{C4} 
  If $a > 0$ or the absolutely continuous component of $K(\d y)$ has an infinite total mass,
  then $\mu_t$ is absolutely continuous for every $t>0$.
\end{cor}
\begin{proof}
  Using \citet[Theorem 27.7 and Lemma 27.1]{Sato}, it follows
  from the assumptions of the statement that
  the one-dimensional distributions of the L\'evy process $\xi(t)$ are 
  absolutely continuous for every $t>0$.
  Our claim follows from the representation in \autoref{T1}.
\end{proof}

\section{The self-similar case}
\label{s:ss}

We now turn our attention to the growth-fragmentation equation \eqref{EqFG} 
for ${\mathcal L}={\mathcal L}_{\alpha}$ given by \eqref{Eqop1} and $\alpha\neq 0$. 
We first point out that the function $\kappa$ is non-increasing if and only if $a=0$,
the dislocation measure $K$ fulfills \eqref{Eqcondnu1}, and
$$
  b+ \int_{[1/2,1)}(1-y) K(\d y) \leq 0.
$$
In this case, the operator ${\mathcal L}_{\alpha}$ can be expressed in the form \eqref{Eqop0} with
$c\leq 0$,  and \eqref{EqFG} is then a pure fragmentation equation as studied by Haas \cite{Haas}. 
To avoid duplication of existing literature, this case will be implicitly excluded hereafter. 

Recall the  notation $h_q(x)=x^q$ for $x>0$. In the self-similar case, power functions are no 
longer eigenfunctions of the operator ${\mathcal L}_{\alpha}$; however, there is the simple relation
\begin{equation}\label{Eqvpss}
  {\mathcal L}_{\alpha}h_q=\kappa(q) h_{q+\alpha}, \qquad q \in \dom \kappa;
\end{equation}
see \autorefpref{L0}{i:L0:eigenfunction}. 
Hence, if \eqref{EqFG} applies to power functions, the linear equation \eqref{EqMelh} for the
Mellin transform $ {M}_t(z) = \langle \mu_t, h_{z-1}\rangle$ in the homogeneous case has
to be replaced by the system
\begin{equation}\label{EqMelss}
  \partial_t {M}_t(1+q) = \kappa(q){M}_t(1+q+\alpha).
\end{equation}

We make  the fundamental assumption, that
\begin{equation}\label{Eqnegspeed}
\inf_{q\geq 0} \kappa(q) < 0,
\end{equation}
which is implicitly enforced throughout this section. 
The role and the importance of \eqref{Eqnegspeed} shall become clear in the sequel.
Recall that $\kappa$ is a convex function on $\R$, and is ultimately increasing,
since we are excluding the case when $\kappa$ is non-increasing throughout \autoref{s:ss}.
Hence, condition \eqref{Eqnegspeed}
ensures the existence of a unique $\omega_{_+}\in\R$ with
\begin{equation}\label{EqMalthus}
\kappa(\omega_{_+})=0 \hbox{ and } \kappa'(\omega_{_+})>0.
\end{equation}
We refer to $\omega_{_+}$ as the \emph{Malthusian parameter}.

The sign of the scaling parameter $\alpha$ plays a crucial role, and we shall study the two cases
separately, even though some ideas are similar.

\subsection{The case \texorpdfstring{$\alpha < 0$}{alpha < 0}}
\label{s:alpha-neg}

We now focus on the case $\alpha <0$.  
We start by observing that the existence of a Malthusian parameter enables us to view \eqref{EqMelss} as a closed
system for an arithmetic  sequence,  and thus 
to solve it.

\begin{lem}\label{P3}
  Consider a sequence of functions 
  ${M}_{\bullet}(1+q)\from [0,\infty)\to (0,\infty)$ for
  \linebreak 
  $q=\omega_{_+}-k\alpha$, $k=-1,0,1, \ldots$,
  with ${M}_{0}(1+q)=1$.
  Suppose that \eqref{EqMelss} and \eqref{Eqnegspeed} hold and recall that $\omega_+$ is the Malthusian parameter 
  defined by \eqref{EqMalthus}. Then ${M}_t(1+\omega_{_+})=  1$ for all $t\geq 0$ and 
  for $k=1,2, \ldots$, we have 
  \[ {M}_t(1+\omega_{_+}-k\alpha)
    = 1 + \sum_{\ell=1}^k \frac{\kappa(\omega_{_+}-\alpha k)\cdots \kappa(\omega_{_+}-\alpha(k-\ell+1)) }{\ell!}\, t^{\ell} . \]
\end{lem}
\begin{proof}
  The equation \eqref{EqMelss} applied to the Malthusian exponent $\omega_{_+}$
  implies that the function $t\mapsto {M}_t(1+\omega_{_+})$ is constant. 
  We can then solve
  \eqref{EqMelss} for $q=\omega_{_+}-\alpha k$ and $k=1, 2, \ldots$ by induction in order to obtain the given formula.
\end{proof}
 
In comparison with  the homogeneous case, \autoref{P3} is a much weaker result than \eqref{SolMelh}, 
as we are not able to compute the whole Mellin transform of a solution, but merely its moments for orders
forming an arithmetic sequence. There is hence an additional crucial issue: it does not suffice
to find a family of measures having the desired moments, but also to ensure that the moment problem is determining. 
It turns out that moment calculations which were performed in \cite{BY-moments}
for self-similar Markov processes enable us  to solve the moment problem in \autoref{P3},
and check that this indeed yields a solution to \eqref{Eqop1}.
Similar calculations also point at a rather surprising result, 
namely that the self-similar growth-fragmentation permits
spontaneous generation!

\begin{thm}\label{T2}
  Assume \eqref{Eqnegspeed} and $\alpha <0$.
  \begin{enumerate}[(i)]
    \item\relax\label{i:T2:pos}
      For every $t\geq 0$, there exists a unique measure $\mu^{ }_t$ on $(0,\infty)$ such that 
      \linebreak 
      $\langle \mu^{ }_t, h_{\omega_{_+}}\rangle=1$ and for every integer $k\geq 1$,
      $$\langle \mu^{ }_t, h_{\omega_{_+}-k\alpha}\rangle
        = 1 + \sum_{\ell=1}^k  
        \frac{\kappa(\omega_{_+}-\alpha k)\cdots \kappa(\omega_{_+}-\alpha(k-\ell+1)) }{\ell!}\, t^{\ell} .$$
      In particular, $\mu^{ }_0=\delta_1$ and the family $(\mu^{ }_t)_{t\geq0}$ 
      solves \eqref{EqFG} for all $f\in{\mathcal C}^{\infty}_c(0,\infty)$
      when ${\mathcal L}={\mathcal L}_{\alpha}$ is given by \eqref{Eqop1}. 

    \item\relax\label{i:T2:zero}
      For every $t>0$, there exists a unique measure $\gamma^{  }_t$ on
      $(0,\infty)$ such that $\langle \gamma^{  }_t, h_{\omega_{_+}}\rangle=1$ and 
      $$\langle \gamma^{  }_t, h_{\omega_{_+}-k\alpha}\rangle
        = t^k\frac{\kappa(\omega_{_+}-\alpha)\cdots \kappa(\omega_{_+}-\alpha k)}{k!}
        \qquad \hbox{for every integer $k\geq 1$}.$$
      If we further set $\gamma^{  }_0\equiv 0$, then the family $(\gamma^{  }_t)_{t\geq0}$ solves \eqref{EqFG} for all $f\in{\mathcal C}^{\infty}_c(0,\infty)$
      when ${\mathcal L}={\mathcal L}_{\alpha}$ is given by \eqref{Eqop1}. 
  \end{enumerate}
\end{thm}
\begin{proof}
  \begin{enumerate}[(i)]
    \item 
      Let us define $\Phi_{_+} \coloneqq \Phi_{\omega_+} = \kappa(\cdot +\omega_+)$,
      which, as we saw in \autoref{L1}, is the Laplace exponent of the Lévy
      process $\xi_{_+} \coloneqq \xi_{\omega_+}$.
      Observe that
      \linebreak 
      $\Phi_{_+}'(0)=\kappa'(\omega_{_+})>0$, 
      so this L\'evy process has a strictly positive and finite first moment.

      Proposition 1 in \cite{BY-moments} then ensures, for every $t>0$,
      the existence of a unique probability measure $\rho^{  }_t$ on $(0,\infty)$,
      such that for every integer $k\geq 1$, 
      $$\langle \rho^{  }_t, h_{-\alpha k}\rangle 
        = 1 + \sum_{\ell=1}^k 
        \frac{\kappa(\omega_{_+}-\alpha k)\cdots \kappa(\omega_{_+}-\alpha(k-\ell+1)) }{\ell!}\, t^{\ell} \,,$$
      so that in particular $\rho^{  }_0=\delta_1$.
      Thus we may set $\mu^{ }_t(\d x) = x^{-\omega_{_+}} \rho^{  }_t(\d x)$,
      and then $\langle \mu^{ }_t, h_{\omega_{_+}-k\alpha}\rangle$ is given
      as in the statement for every integer $k\geq 0$. That this determines $\mu^{ }_t$ 
      derives from the uniqueness of $\rho^{  }_t$. 

      Now, using \eqref{Eqvpss}, we immediately check that 
      $\langle \mu_t, h_{\omega_{_+}-k\alpha}\rangle$ satisfies \eqref{EqMelss}.
      It then follows   that $(\mu^{ }_t)_{t\geq 0}$ solves \eqref{EqFG} for every $f\in{\mathcal C}^{\infty}_c(0,\infty)$
      (recall that the probability measure $\rho^{  }_t$ is determined by its entire moments).
      Finally, the map $t\mapsto \langle \rho^{  }_t, h_{-\alpha k}\rangle$ is continuous, and we deduce that
      $(\rho^{  }_t)_{t\geq 0}$ is vaguely continuous (using again the fact that $\rho^{  }_t$ is determined 
      by its moments $\langle \rho^{  }_t, h_{-\alpha k}\rangle$ for $k\in\N$).
      Hence the same holds for $(\mu^{ }_t)_{t\geq 0}$. 
    \item
      Recall from above that $\Phi_{_+}=\kappa(\omega_{_+}+\cdot)$ is the Laplace exponent of a
      spectrally negative L\'evy process which has strictly positive and finite first moments.
      Proposition 1 in \cite{BY-moments} ensures for every $t>0$  the existence of a unique probability measure 
      $\pi^{  }_t$ on $(0,\infty)$ such that its
      moments  are given by 
      \[ \langle \pi^{  }_t, h_{-\alpha k}\rangle 
        = t^k\frac{\Phi_{_+}(-\alpha)\cdots \Phi_{_+}(-\alpha k)}{k!}
        = t^k\frac{\kappa(\omega_{_+}-\alpha)\cdots \kappa(\omega_{_+}-\alpha k)}{k!}, \]
      for $k=1,2, \dotsc$, 
      and this determines $\pi^{  }_t$. 
      It follows immediately that $(\pi^{  }_t)_{t\geq 0}$ is vaguely continuous (recall that $\pi^{  }_0=0$).

      We then define for $t>0$
      $$\gamma^{  }_t(\d x) = x^{-\omega_{_+}} \pi^{  }_t(\d x)\,, \qquad x>0,$$
      so 
      $$\langle \gamma^{  }_t, h_{\omega_{_+}-k\alpha}\rangle
        = t^k\frac{\kappa(\omega_{_+}-\alpha)\cdots \kappa(\omega_{_+}-\alpha k)}{k!}.$$
      Then  \eqref{Eqvpss} entails that for every integer $k\geq 1$, there is the identity 
      \begin{eqnarr*}
        \partial_t\langle \gamma^{  }_t, h_{\omega_{_+}-k\alpha}\rangle
        &=& 
        t^{k-1}\frac{\kappa(\omega_{_+}-\alpha)\cdots \kappa(\omega_{_+}-\alpha k)}{(k-1)!}\\
        &=&
        \kappa(\omega_{_+}-\alpha k)\langle \gamma^{  }_t, h_{\omega_{_+}-(k-1)\alpha}\rangle
        = \langle \gamma^{  }_t, \mathcal{L}h_{\omega_{_+}-k\alpha}\rangle, 
      \end{eqnarr*}
      and the conclusion follows just as in (i). \qedhere
  \end{enumerate}
\end{proof}

\autorefpref{T2}{i:T2:zero} entails that uniqueness of the solution fails
when one only requires \eqref{Eqop1} to be fulfilled for all
$f\in{\mathcal C}^{\infty}_c(0,\infty)$,
which contrasts sharply with the results of Haas \cite{Haas} for the pure-fragmentation equation.
We conjecture that the solution $\mu^{ }_t$ given in \autorefpref{T2}{i:T2:pos} is minimal,
in the sense that if $(\tilde \mu_t)_{t\geq 0}$ is another solution
with the same initial condition $\tilde \mu_0=\delta_1$,
then $\mu^{ }_t \leq \tilde \mu_t$ for every $t>0$. 
We also stress that uniqueness of the solution can be restored by 
requiring \eqref{EqFG} to hold for the functions $h_q$ with $q\geq \omega_{_+}+\alpha$;
see \autoref{P3} and \autorefpref{T2}{i:T2:pos}.

\skippar
We now present a different approach to \autoref{T2}.
In the homogeneous case $\alpha=0$, we saw in the preceding section 
that the equation \eqref{Eqop1} bears a close relationship with certain exponential L\'evy processes. 
It turns out that in the self-similar case with $\alpha<0$, the vital connection
is with positive 
self-similar Markov processes, and is made via the Lamperti transformation which associates these
with the class of L\'evy processes. We first provide some background in this area. 

A \define{positive self-similar Markov process} (\define{pssMp}) with
\define{self-similarity index} $\gamma \in \R$ is a standard Markov process
$R = (R_t)_{t\geq 0}$ with associated filtration $\FFt$ and probability laws
$(\ssP_x)_{x \in (0,\infty)}$, on $[0,\infty]$, which has $0$ and $\infty$ as
absorbing states and
which satisfies the \define{scaling property}, that for every $x, c > 0$,
\[
  \label{scaling prop}%
  \text{ the law of } (cR_{t c^{-\alpha}})_{t \ge 0}
  \text{ under } \ssP_x \text{ is } \ssP_{cx} \text{.}
\]
Here, we mean ``standard'' in the sense of \cite{BG-mppt},
which is to say, $\FFt$ is a complete, right-continuous filtration,
and $R$ has c\`adl\`ag paths and is strong Markov
and quasi-left-continuous.

In the seminal paper \cite{Lamp}, Lamperti describes a one-to-one correspondence
between pssMps and (possibly killed) L\'evy processes, which we now outline.
It may be worth noting that we have presented a slightly
different definition of pssMp from Lamperti; for the connection, see
\cite[\S 0]{VA-Ito}.

Let
$S(t) = \int_0^t (R_u)^{-\gamma}\, \dd u .$
This process is continuous and strictly increasing until $R$ reaches zero.
Let $(T(s))_{s \ge 0}$ be its inverse, and define
\[ \eta_s = \log R_{T(s)} \qquad s\geq 0. 
\]
Then $\eta : = (\eta_s)_{s\geq 0}$ is a (possibly killed) L\'evy process started
at position $\log x$, possibly killed at an independent
exponential time; the law of the L\'evy process and the
rate of killing do not depend
on the value of $x$. The real-valued process
$\eta$ with probability laws
$(\LevP_y)_{y \in \R}$ is called the
\define{L\'evy process associated to $R$},
or the \define{Lamperti transform of $R$}.

An equivalent definition of $S$ and $T$, in terms of $\eta$ instead
of $R$, is given by taking
$T(s) = \int_0^s \exp(\gamma \eta_u)\, \dd u$
and $S$ as its inverse. Then,
\begin{equation*}
  \label{Lamp repr}
  R_t = \exp(\eta_{S(t)}) 
\end{equation*}
for all $t\geq 0$, and this shows that the Lamperti transform is a bijection.
A useful fact is that, as a consequence of the definitions we have just given,
it holds that $\dd t = \exp(-\gamma \eta_{S(t)}) \, \dd S(t)$.

Most of the literature on pssMps (including Lamperti's
original paper) assumes that $\gamma>0$, and much of it is also given for $\gamma=1$.
Indeed it is easy to change the index of self-similarity.
If $R$ is a pssMp of index $\gamma$ and corresponding to the
Lévy process $\eta$, then, for any
$\gamma' \in \R$, the process $R^{\gamma'} = (R_t^{\gamma'})_{t\ge 0}$
is a pssMp with index $\gamma/\gamma'$, corresponding to
the Lévy process $\gamma' \eta$. It is also useful to note that the
time-changes appearing in the Lamperti transformation are a.s.\ equal
for $R$ and $R^{\gamma'}$.
We should point out that the case $\gamma=0$ is special, since
in this case the time-change does not have any effect, and the
pssMps of index $0$ are just exponential Lévy processes.

Note that, if the Lévy process process $\eta$ is killed at time $\zeta$, then we define
$R_t = 0$ for $t \ge T(\zeta)$ if $\gamma \ge 0$, and
$R_t = +\infty$ for $t \ge T(\zeta)$ if $\gamma < 0$.

Recall \autoref{L1}, and  define $\Phi_{_+} \coloneqq \Phi_\omega =\kappa(\cdot+\omega_{_+})$,
which is the Laplace exponent of the spectrally negative L\'evy process $\xi_{_+} \coloneqq \xi_{\omega_+}$.
Let us denote by $X_+$ the pssMp with index $-\alpha$ associated to $\xi_+$ by the Lamperti
transformation. Note that, because $\xi_+$ has positive mean, the process $X_+$ never
reaches the absorbing boundaries $0$ or $+\infty$.
We define the measure $\rho_t^{  }$ to be the distribution of $X_{_+}(t)$ under $\ssP_1$, that is,
 the probability measure on $(0,\infty)$ defined by
$$\langle \rho_t^{  }, f \rangle = \ssE_1(f(X_{_+}(t)))\,,\qquad f\in{\mathcal C}_0(0,\infty)$$
and give first the following analogue of \autoref{C0}:

\begin{lem}\label{LanC0}
  The family of probability measures $(\rho^{  }_t)_{t\geq 0}$ 
  depends continuously on the parameter $t$ for the topology of weak convergence.
  Further, for every $g\in {\mathcal C}^{\infty}_c(0,\infty)$, 
  the function $t\mapsto \langle \rho^{  }_t, g\rangle$ is differentiable with derivative 
  $\partial_t\langle \rho^{  }_t, g\rangle= \langle \rho^{  }_t, {\mathcal A}^{(\alpha)}_{_+}g\rangle$,
  where 
  $${\mathcal A}^{(\alpha)}_{_+}g(x) \coloneqq x^{-\omega_{_+}} {\mathcal L}_{\alpha}(h_{\omega_{_+}}g)(x)\,, \qquad x>0.$$
\end{lem}
\begin{proof}
  The first assertion follows easily from the Feller property of self-similar Markov processes; 
  see Theorem 2.1 in Lamperti \cite{Lamp} and the remark on page 212.
  In order to establish the second, we work with the L\'evy process $\xi_{_+}$
  The exponential L\'evy process $\exp(\xi_{_+}(\cdot))$ is a Feller process in $(0,\infty)$, 
  and the same calculation as in the proof of \autoref{C0} shows that its infinitesimal generator
  ${\mathcal A}_{_+}$ is given by 
  $${\mathcal A}_{_+}g(x) 
  = x^{-\omega_{_+}} {\mathcal L}_{0}(h_{\omega_{_+}}g)(x)\,,\qquad g\in{\mathcal C}^{\infty}_c(0,\infty).$$

  According to Dynkin's formula (see, e.g., Proposition 4.1.7 of \cite{EK-mp}),
  for every 
  \linebreak 
  $g\in{\mathcal C}^{\infty}_c(0,\infty)$, the process
  $$g(\exp(\xi_{_+}(t)))- \int_0^t {\mathcal A}_{_+}g(\exp(\xi_{_+}(s))) \d s$$
  is a martingale. Recall that by definition, $X_{_+}$ arises as the transform of $\exp(\xi_{_+}(\cdot))$
  by the time substitution ${S}$, which is given as the inverse of  the additive functional
  $\int_0^t h^{-1}_{\alpha}\left( \exp(\xi_{_+}(s)\right)\d s $, and we have the identity 
  $$
    g(X_{_+}(t))- \int_0^{{S}(t)} {\mathcal A}_{_+}g(\exp(\xi_{_+}(s))) \d s
    = g(X_{_+}(t))- \int_0^{t} h_{\alpha}(X_{_+}(s)) {\mathcal A}_{_+}g(X_{_+}(s)) \d s.
  $$
  A priori, the time-substitution above changes a martingale into a local martingale. 
  However, using \autorefpref{Lbounds}{i:Lbounds:asym} and the fact that $\alpha <0$, we see that ${\mathcal A}^{(\alpha)}_{_+}g\coloneqq h_{\alpha} {\mathcal A}_{_+}g$ is bounded, and it follows that the process 
  $$g(X_{_+}(t))- \int_0^{t}  {\mathcal A}^{(\alpha)}_{_+}g(X_{_+}(s)) \d s$$
  is a true martingale. Taking expectations, we arrive at 
  $$\langle \rho^{  }_t, g\rangle - g(1)=\int_0^t \langle \rho^{  }_s,  {\mathcal A}^{(\alpha)}_{_+}g\rangle\d s,$$
  and our claim follows.
\end{proof}

The connection with solutions of the growth-fragmentation equation is the following:
\begin{cor}\label{c:mu-alpha-neg}
  Let
  \[ \tilde \mu^{ }_t = h_{-\omega_{_+}}\rho_t^{  }, \qquad t\geq 0. \]
  Then, $(\tilde \mu^{ }_t)_{t\geq 0}$ is 
  equal to the solution $(\mu^{ }_t)_{t \ge 0}$ of the growth-fragmentation equation appearing in \autorefpref{T2}{i:T2:pos}.
\end{cor}
\begin{proof}
  We deduce immediately from \autoref{LanC0} that for every $f\in{\mathcal C}^{\infty}_c$
  $$\partial_t\langle \tilde\mu^{ }_t, f\rangle
    = \partial_t\langle \rho^{  }_t, h^{-1}_{\omega_{_+}}f\rangle
    = \langle \rho^{  }_t, h^{-1}_{\omega_{_+}}{\mathcal L}_{\alpha}f\rangle
    =  \langle \tilde\mu^{}_t, {\mathcal L}_{\alpha}f\rangle\,,$$
  that is, the family $\bigl(\tilde\mu^{ }_t\bigr)_{t\geq 0}$  solves 
  \eqref{EqFG} with ${\mathcal L}={\mathcal L}_{\alpha}$. 
  That $\mu^{ }_t$ coincides with the measure appearing in \autorefpref{T2}{i:T2:pos},
  and that the notation 
  $\rho_t^{  }$ for the distribution of $X_{_+}(t)$ is consistent with that used in the proof of \autorefpref{T2}{i:T2:pos},
  follows from Proposition 1 of \cite{BY-moments}.
\end{proof}

This approach could also be adapted to prove the existence of $(\gamma_t^{  })_{t\geq 0}$
using the process $X_+(t)$ started from zero, and indeed, this will be our method for
the case $\alpha>0$ in \autoref{s:alpha-pos}.

\skippar
We conclude the section by offering some results on the asymptotic behaviour of the solution
$(\mu_t)_{t\ge 0}$ given by the previous theorem.
  
Our first result in this direction indicates that, thanks to the self-similarity property \eqref{EqSSL} of
the equation \eqref{EqFG}, the solution starting
from zero given above can
be used to describe the asymptotic behaviour of $\mu^{ }_t$ as $t\to\infty$.

\begin{prop}\label{p:asymptotic-negative}
  For any $f \in \Cb$,
  \[ \int f(t^{-1/\abs{\alpha}} z) z^{\omega_+} \mu_t^{  }(\dd z) \to \int f(z) z^{\omega_+} \gamma_1^{  }(\dd z) . \]
\end{prop}
  \begin{proof}
   
    Since $\kappa'(\omega_+) >0$, it is possible to extend 
    the definition of $\Xp$ in order to allow it to start
    from $\Xp(0)=0$, such that it is a Feller process on the state space $[0,\infty)$;
    this is a consequence of \cite[Theorem 1]{BY-entrance}.
    For $x \ge 0$, we will denote by $\ssP_x$ the law of the process with $\Xp(0) = x$.

    It then follows from the scaling property that $\ssE_x[ f(t^{1/\alpha} \Xp(t))] = \ssE_{x t^{1/\alpha}}[f(\Xp(1))]$,
    and then the convergence
    \[ \ssE_x[f(t^{1/\alpha} \Xp(t)) ] \to \ssE_0[f(\Xp(1))], \qquad t \to \infty, \]
    follows from the scaling property of $\Xp$.
     
    Finally, we know from the reference \cite{BY-moments}, which we used in 
    the proof of \autoref{T2}, that
    the measures $x^{\omega_+}\mu_t^{  }(\dd x)$ and $x^{\omega_+} \gamma_t^{  }(\dd x)$ are,
    respectively, equal to $\ssP_1(X_+ (t) \in \dd x)$ and $\ssP_0(X_+(t) \in \dd x)$.
    Our claim follows immediately.
  \end{proof}
  
We remark that the statement of the proposition can easily be extended to solutions
based on $(\mu_t^{  })$ whose initial value is a measure with compact support
in $(0,\infty)$.

\skippar
Suppose now that the equation $\kappa(q) = 0$ has two solutions,
$\omega_{_-}$ and $\omega_+$, with $\omega_{_-}<\omega_+$. Then we can say a little more.
Let $X_-$ be the $(-\alpha)$-pssMp associated with the Lévy process
$\xi_- \coloneqq (\xi_-(t))_{t\ge 0}$
having Laplace exponent $\Phi_- \coloneqq \Phi_{\omega_{_-}}$. 
Recall that we say the Lévy process $\xi_-$ is \emph{lattice} if,
for some $r \in \R$, the support 
of $\xi_-(1)$ a.s.\ lies in $r\ZZ$;
otherwise, we say that $\xi_-$ is \emph{non-lattice}.
If we define
the random variable
\[ I = \int_0^\infty \e^{\abs{\alpha}\xi_-(t)}\, \dd t, \]
then it is known from \cite[Lemma 4]{Riv-re1} that, so long as
$\xi_-$ is non-lattice,
\[ \lim_{t\to \infty} t^{(\omega_+-\omega_{_-})/\abs{\alpha}} \P_0(I>t) = C , \]
for some $0 < C < \infty$.
We obtain from \citet{HR-Yaglom} the following result.
\begin{prop}
  Let $f \in \mathcal{C}_0(0,\infty)$ and assume that $\xi_-$ is non-lattice. Then,
  \[ \dfrac{\dint x^{\omega_{_-}} f(t^{-1}x^{\abs{\alpha}}) \, \mu_t^{  }(\dd x)}{\P_0(I>t)}
    \to \int f(x) \, \nu(\dd x), \qquad t \to \infty,
  \]
  where $\nu$ is the distribution of the random variable
  $J_{(\omega_+-\omega_{_-})/\abs{\alpha}}$ in equation (13) of \cite{HR-Yaglom}.
\end{prop}
  \begin{proof}
    As remarked in the proof of \autoref{p:asymptotic-negative}, the equation
    \[ x^{\omega_+}\mu_t^{  }(\dd x) = \ssP_1(X_+(t) \in \dd x) \] holds as an identity of
    probability measures, where $X_+$ is the $(-\alpha)$-pssMp corresponding to
    the Lévy process $\xi_+$ with Laplace exponent $\Phi_+ = \Phi_{\omega_+}$.
    We now wish to use the `Esscher transform' for pssMps, which is essentially
    obtained by standard arguments from the Esscher transform of Lévy processes
    given in \cite[Theorem 3.9]{Kypr2} (see, for instance, the discussion around
    \cite[Theorem 14]{CR-cond} for an application in the context of pssMps.)
    This allows
    us to perform a change of measure to switch from the
    process $X_+$, related to the Laplace exponent $\Phi_+$,
    to the process $X_-$, related to the Laplace exponent $\Phi_- = \Phi_+(\cdot + \omega_{_-} - \omega_+)$.
    Specifically, we have:
    \[ x^{\omega_{_-}}\mu_t^{  }(\dd x) 
      = x^{\omega_{_-}-\omega_+}x^{\omega_+}\mu_t^{  }(\dd x)
      = x^{\omega_{_-}-\omega_+} \ssP_1(X_+(t) \in \dd x)
      = \ssP_1(X_-(t) \in \dd x),
    \]
    for $x > 0$.
    The process
    $\xi_-$ (which corresponds to $X_-$) is a Lévy process with non-monotone paths
    and which satisfies the conditions of \citet[Theorem~1.6]{HR-Yaglom}. Applying
    this theorem gives the result.
  \end{proof}

\subsection{The case \texorpdfstring{$\alpha >0$}{alpha > 0}}
\label{s:alpha-pos}

In the case $\alpha>0$, the equation \eqref{EqMelss} for the Mellin transform is unfortunately
much less useful, for the following reasons. 
Firstly, the analogue of \autoref{P3} would require to assume that 
$\langle \mu_t,h_q\rangle <\infty$ for all $q$ sufficiently negative.
Roughly speaking, this would force the scarcity of small particles, 
and this phenomenon only occurs for a very restricted class of dislocation measures $K$
(informally, dislocations should not generate too many small particles, and in particular
the total intensity of dislocations must be finite).
Secondly, even if one were able to get an expression for the (negative) moments 
$\langle \mu^{ }_t, h_{\omega_{_+}-k\alpha}\rangle$ with $k\in\N$, 
this moment problem would be in general indeterminate, and the arguments used
in the preceding section would thus collapse.

Nonetheless, we have just seen from \autoref{LanC0} that, for $\alpha<0$,
self-similar growth-fragmentation equations have a close connection
with certain self-similar Markov processes, and using the intuition that we gained,
we are able to offer a very similar set of results when $\alpha > 0$. 
Recall that $\kappa\from [0,\infty)\to (-\infty,\infty]$ is a convex function, 
and that, since we are assuming \eqref{Eqnegspeed},
we may pick $\omega>0$ such that $\mathtt{k}\coloneqq -\kappa(\omega)>0$. 
Then, the function 
$$\Phi_{\dagger}(q)\coloneqq \Phi_{\omega}(q) + \mathtt{k}= \kappa(\omega+q) \,,\qquad q\geq 0,$$
is the Laplace exponent of a spectrally negative L\'evy process killed at an independent exponential
time with parameter $\mathtt{k}$, say 
$\xi_{\dagger}$, and we denote by $X_{\dagger}$ the $(-\alpha)$-pssMp
associated with $\xi_\dagger$ via the Lamperti transformation. 
Note that $X_{\dagger}$ hits the absorbing state $+\infty$ by a jump.
 
We write $\rho_t^{  }$
for the sub-probability measure on $(0,\infty)$ induced by the distribution of $X_{\dagger}(t)$
and study its properties in the following result, which
mirrors \autoref{LanC0}.

\begin{lem}\relax\label{LanC0-}
  Suppose \eqref{Eqnegspeed} holds and $\alpha >0$. Then we have, in the notation above:
  \begin{enumerate}[(i)]
    \item\relax\label{i:LanC0-:sup}
      $\E( \sup_{t \ge 0} X_{\dagger}(t)^q) <\infty$
      for all $0\leq q <\omega_{_+}-\omega$. 
  \item\relax\label{i:LanC0-:ef}
      $\E\biggl[ \displaystyle\int_0^\infty X_{\dagger}(u)^p \, \dd u \biggr] < \infty$
      for all
      $0 < p -\alpha < \omega_{_+} - \omega$.
  \item\relax\label{i:LanC0-:sol}
    The family  $(\rho^{  }_t)_{t\geq 0}$ depends continuously on the parameter $t$ for the topology
    of weak convergence.
    For every
    $g\in {\mathcal C}^{\infty}_c(0,\infty)$, the function $t\mapsto \langle \rho^{  }_t, g\rangle$
    is differentiable with derivative 
    $\partial_t\langle \rho^{  }_t, g\rangle= \langle \rho^{  }_t, {\mathcal A}^{(\alpha)}_{\dagger}g\rangle$,
    where 
    \[ {\mathcal A}^{(\alpha)}_{\dagger}g(x)
      \coloneqq x^{-\omega} {\mathcal L}_{\alpha}(h_{\omega}g)(x)\,, \qquad x>0.
    \]
  \item
    $(\rho^{  }_t)_{t\ge 0}$ solves the above equation also for $g(x) = x^q$ with $0 < q < \omega_{_+} - \omega$.
\end{enumerate}
\end{lem}
\begin{proof}
  \begin{enumerate}[(i)]
    \item 
      From the very construction of $X_{\dagger}$, the overall supremum
      \linebreak 
      $\bar X_{\dagger}\coloneqq \sup_{t\geq 0} X_{\dagger}(t)$ is given by $\bar X_{\dagger}=\exp(\bar\xi_{\dagger})$, 
      with $\bar\xi_{\dagger}\coloneqq \sup_{t\geq 0} \xi_{\dagger}(t)$. We infer from Corollary VII.2 
      in \cite{BeLP} that $\bar\xi_{\dagger}$ has the exponential distribution with parameter $\omega_{_+}-\omega$
      (which is the positive root to the equation $\Phi_{\dagger}(q)=0$), and our claim follows.
    \item
      We begin with the following calculation, using the discussion of pssMps in the preceding section.
      Recall that $S$ is the time-change appearing in the Lamperti transform relating $X_{\dagger}$ and $\xi_{\dagger}$, and that there is the identity $\dd S(t)=\exp(\alpha \xi_{\dagger}(t))\dd t$. We have therefore  
      \begin{align*}
        \int_0^\infty X_{\dagger}(u)^p \, \dd u 
        &= \int_0^\infty \e^{p\xi_{\dagger}(S(u))} \, \dd u \\
        &= \int_0^\infty \e^{(p-\alpha)\xi_{\dagger}(S(u))} \, \dd S(u)
        = \int_0^{\infty} \e^{(p-\alpha)\xi_{\dagger}(s)} \, \dd s.
     \end{align*}
     But now we can consider the expectation:
     \begin{align*} \E_x\biggl[ \int_0^\infty X_{\dagger}(u)^p \, \dd u \biggr]
       &= x^{p-\alpha} \E\biggl[ \int_0^\infty \e^{(p-\alpha)\xi_{\dagger}(s)}\, \dd s \biggr] \\
       &= \begin{cases}
            \dfrac{x^{p-\alpha}}{\kappa(p-\alpha+\omega)}, & \text{if } \kappa(p-\alpha+\omega) < 0, \\
            \infty, & \text{otherwise}.
          \end{cases}
     \end{align*}
     We complete the proof by recalling the definition of  $\omega_{_+}$.
    \item
      Just as in the proof of \autoref{LanC0}, the first assertion follows from  the
      Feller property of self-similar Markov processes, and  the process
      \begin{equation}\label{e:time-change-local-martingale}
        N_t \coloneqq g(X_{\dagger}(t))-g(1)-\int_0^t {\mathcal A}^{(\alpha)}_{\dagger}  g(X_{\dagger}(s)) \d s
      \end{equation}
      is a local martingale for every $g\in{\mathcal C}^{\infty}_c(0,\infty)$.
      We will show that \[ \E\bigl(\textstyle\sup_{s \le t} \abs{N_s}\bigr) < \infty, \qquad t \ge 0, \]
      which implies that $N$ is a true martingale;
      see \cite[Theorem I.51]{Pro-stoch-calc}.
    
     Since $g$ is bounded, certainly $\sup_{s\le t}\abs{g(X_{\dagger}(s))-g(x)}$ is in
     $L^1(\P)$. In constrast to \autoref{LanC0}, the function
     $\mathcal{A}^{(\alpha)}_{\dagger}g$ may be unbounded for $\alpha > 0$;
     however, we do know from 
     \autorefpref{Lbounds}{i:Lbounds:asym} that ${\mathcal A}^{(\alpha)}_{\dagger}g$
     is zero on some neighborhood of $0$ and,
     for any $q \in \dom \kappa$, fulfills ${\mathcal A}^{(\alpha)}_{\dagger}g= o(x^{q+\alpha-\omega})$ 
     as $x\to \infty$.
    
     We let $\omega < q < \omega_{_+}$ and keep
     it fixed for the rest of the proof.
     For some $K>0$ we then have
     \begin{align}
       \E \biggl[\sup_{u\le t} \abs[\bigg]{\int_0^u {\mathcal A}^{(\alpha)}_{\dagger}g(X_{\dagger}(s))\, \dd s} \biggr]
       &\le
       \E \biggl[\int_0^t \abs[\big]{ {\mathcal A}^{(\alpha)}_{\dagger}g(X_{\dagger}(s))}\,\dd s \biggr] \label{e:mg-crit-ef} \\
       &\le t \sup_{[0,K]} \abs{{\mathcal A}^{(\alpha)}_{\dagger}g}
       +  \E \biggl[ \int_0^t X_{\dagger}(s)^{q+\alpha-\omega} \, \Indic{X_{\dagger}(s) > K} \, \dd s \biggr]. \nonumber
     \end{align}
     Setting $p = q+\alpha-\omega$ in part \pref{i:LanC0-:ef}, we see that the right-hand
     side is finite.
    
     We have thus shown that $N$ is a true martingale, and
     \begin{equation*}\label{e:Ag-Fubini}
       \E\left( \int_0^t |{\mathcal A}^{(\alpha)}_{\dagger}g(X_{\dagger}(s)) |\d s \right)<\infty.
     \end{equation*}
     Taking expectations in \eqref{e:time-change-local-martingale}
     and applying Fubini's theorem, we obtain
     \[
       \langle \rho^{  }_t, g\rangle - g(1)=\int_0^t \langle \rho^{  }_s, 
       \mathcal{A}^{(\alpha)}_{\dagger}g\rangle\d s, 
     \]
     which completes the proof.
    \item
      This part is proved by setting $g(x) = x^q$ in the previous part, as follows.
      Using the Markov property one sees immediately that the process
      \[ M_t = \e^{q\xi_{\dagger}(t)} - 1 - \kappa(\omega_{_-}+q)\int_0^t \e^{q\xi_{\dagger}(s)}\,\dd s, \qquad t \ge 0, \]
      is a martingale in the filtration of $\xi_{\dagger}$ for every $q \ge 0$.
      Applying the same reasoning with the time-change as in \autoref{LanC0}, it follows
      that
      \[
        N_t = X_{\dagger}(t)^q - 1 - \kappa(\omega+q)\int_0^t X_{\dagger}(s)^{q+\alpha}\, \dd s, \qquad t \ge 0,
      \]
      is a local martingale. 
      (For our choice of $g$, we have
      ${\mathcal A}^{(\alpha)}_{\dagger}g(x) = \kappa(\omega+q) x^{q+\alpha}$,
      so this is consistent with the proof of part (iii).)
      We observe that
      \[ \sup_{t \ge 0} \abs{N_t} 
        \le 
        1 + \sup_{t \ge 0} X_{\dagger}(t)^q
        - \kappa(\omega+q) \int_0^\infty X_{\dagger}(s)^{q+\alpha} \, \dd s .
      \]
      We now apply directly parts (i) and (ii) of this lemma in order to show
      that $\E [ \sup_{t \ge 0} \abs{N_t}] < \infty$. This is a sufficient criterion
      for $N$ to be a uniformly integrable martingale (see \cite[Theorem I.51]{Pro-stoch-calc}),
      and the remainder of the proof follows in the same way as in part (iii). \qedhere
  \end{enumerate}
\end{proof}

We can then repeat the calculations which were made after the proof of \autoref{LanC0},
and arrive at:

\begin{cor} \label{c:mu-alpha-pos}
  Suppose that \eqref{Eqnegspeed} holds.
  Define  $\mu^{ }_t = h_{-\omega}\rho_t^{  }$, for $t\geq 0$.
  Then the vaguely continuous family of measures $(\mu^{ }_t)_{t\geq 0}$ solves 
  \eqref{EqFG} with ${\mathcal L}={\mathcal L}_{\alpha}$ both for
  \linebreak 
  $f \in \Ctest(0,\infty)$ and
  for $f=h_q$, for any $\omega<q<\omega_+$.
\end{cor}

An interesting contrast with the case $\alpha < 0$
is that we do \emph{not} show that $\mu_t^{  }$
solves \eqref{EqFG} for all power functions.

\medskip\noindent
We now give the basis of a solution to the growth-fragmentation equation starting from
the zero measure; in this case, the solution should be interpreted, not as spontaneous generation
from infinitely small masses, but as starting from infinite
mass and breaking apart instantaneously. Recall that the equation $\kappa(q)=0$ has at most two solutions.
More precisely, we have already seen that there is always a unique solution
$\omega_{_+}$ with $\kappa'(\omega_{_+}) > 0$ 
(this is the Malthusian exponent defined by \eqref{EqMalthus}).
When a second solution, say $\omega_{_-}$, exists,
then $\omega_{_-}<\omega_{_+}$ and $\kappa'(\omega_{_-}) \in[-\infty, 0)$. 
We give the results under the assumptions:
\begin{equation}\label{EqMalthus-}
  \hbox{
    the equation $\kappa(q)=0$ with $q\geq 0$
    has two solutions $\omega_{_-}<\omega_{_+}$, and $\kappa'(\omega_{_-})>-\infty$
  }
\end{equation}
which is thus stronger than \eqref{EqMalthus}. 
We write $\xi_-$ for the spectrally negative L\'evy process with Laplace exponent
$\Phi_-(q)\coloneqq \Phi_{\omega_-} = \kappa(q+\omega_{_-})$, and then $X_-$ for the pssMp
with index $-\alpha$ associated to $\xi_-$ by Lamperti's transform. 

\begin{lem}\label{Lentrance}
  Assume that \eqref{EqMalthus-} holds.
  Then there exists a c\`adl\`ag process $({\mathcal X}(t))_{t> 0}$ with values in 
  $(0,\infty)$ and $\lim_{t\to 0+}{\mathcal X}(t)=\infty$ a.s.\ such that:
  \begin{itemize}
    \item
      For every $s>0$, conditionally on ${\mathcal X}(s)=x$, the shifted process $({\mathcal X}(s+t))_{t\geq 0}$
      has the law $\ssP_x$ of the pssMp $X_-$ started from $x$.
  
    \item
      For all $0< \varepsilon < (\omega_+-\omega_{_-})/\alpha$, there is $c(\varepsilon)<\infty$ such that 
      $$\E\left({\mathcal X}(t)^{\alpha(1-\varepsilon)}\right) = c(\varepsilon) t^{\varepsilon-1},\qquad t>0.$$
  \end{itemize}
\end{lem}
\begin{proof}
  Let $Y$ denote the pssMp  with self-similarity index $\alpha$ associated to
  the L\'evy process $-\xi_-$, so $Y$ has the same law as $1/X_-$.
  Because \[ \E(-\xi_-(1))=-\Phi'_-(0+)\coloneqq m\in(0,\infty), \]
  \cite{BY-entrance} shows that $0+$ is an entrance boundary for $Y$, 
  i.e., there exists a c\`adl\`ag process $({\mathcal Y}(t))_{t> 0}$
  with values in $(0,\infty)$ and $\lim_{t\to 0+}{\mathcal Y}(t)=0$ a.s., such that, 
  for every $s>0$, conditionally on ${\mathcal Y}(s)=y$, the shifted process
  $({\mathcal Y}(s+t))_{t\geq 0}$ has the law $Y$ started from $y$.
  Our first claim follows by setting $\mathcal{X}(t) = 1/\mathcal{Y}(t)$.
 
  Further, according to Theorem 1 in \cite{BY-entrance}, there is the identity
  $$\E\left( {\mathcal Y}^{\alpha(\varepsilon-1)}(t)\right) = \frac{1}{\alpha m}\E\left(I^{-1}(t/I)^{\varepsilon-1}\right)$$
  where $I\coloneqq \int_0^{\infty}\exp(\alpha\xi_-(s))\dd s$. It thus follows
  $$\E\left( {\mathcal X}^{\alpha(1-\varepsilon)}(t)\right) = c(\varepsilon) t^{\varepsilon-1}$$
  where $c(\varepsilon)=  \E\left(I^{\varepsilon}\right)/ \alpha m \in (0,\infty]$.
 
  For every $0< \varepsilon < (\omega_+-\omega_{_-})/\alpha$, the Laplace exponent 
  $q\mapsto \Phi_-(\alpha q)$ of $\alpha \xi_-$ fulfills $\Phi_-(\alpha \varepsilon)<0$, and according to 
  Lemma 3 in Rivero \cite{Rivero}, this ensures that $\E\left(I^{\varepsilon}\right)<\infty$.
\end{proof}

We next further require that 
$\omega_{_-} \in (\dom \kappa)^\circ$. This is only a little stronger than the
condition $\kappa'(\omega_{_-})>-\infty$, which is necessary and sufficient for $\mathcal X$ to exist.
We write  $\Aa_-$ for the operator $\Aa_{\dagger}$ given in \autorefpref{LanC0-}{i:LanC0-:sol} for
$\omega=\omega_{_-}$, and deduce the following.

\begin{cor}\label{Centrance}
  Assume that \eqref{EqMalthus-} holds and that $\omega_{_-} \in (\dom \kappa)^\circ$. 
  For $t>0$, write $\pi_t$ for the distribution of ${\mathcal X}(t)$. Then, for every $f\in{\mathcal C}^{\infty}_c(0,\infty)$, we have 
  $$\int_0^t \abs[\big]{\langle \pi_s, \Aa_-f\rangle } \, \dd s <\infty,$$
  and there is the identity
  $$\langle \pi_t, f\rangle= \int_0^t \langle \pi_s, \Aa_-f\rangle \, \dd s.$$
\end{cor}
\begin{proof}
  Recall from \autorefpref{LanC0-}{i:LanC0-:sol} that
  $h_{\omega_{_-}}\Aa_-f={\mathcal L}_{\alpha}(h_{\omega_{_-}}f)$,
  so \autorefpref{Lbounds}{i:Lbounds:asym} and the assumption that
  $\omega_{_-}-\alpha \varepsilon \in \dom \kappa$ for some $\varepsilon>0$ entail 
  $$\Aa_-f(x) = o\Bigl(x^{-\omega_{_-}+\alpha+\omega_{_-}-\alpha \varepsilon} \Bigr) = o\Bigl(x^{\alpha(1-\varepsilon)} \Bigr) .$$
  It now follows from the preceding lemma that
  $|\langle \pi_s, \Aa_-f\rangle | \leq C(s^{\varepsilon-1}+1)$, where $C$ is a finite constant
  depend ending only on $f$ and $\varepsilon$. Our first claim follows. 

  Recall then from the proof of \autorefpref{LanC0-}{i:LanC0-:sol} that 
  $$f(X_-(t))-\int_0^t \Aa_-f(X_-(s))\dd s\,,\qquad t\geq 0$$
  is a local martingale, and thus, thanks to \autoref{Lentrance}, so is
  $$f({\mathcal X}(t+s))-f({\mathcal X}(s))-\int_0^t \Aa_-f({\mathcal X}(r+s))\dd r\,,\qquad t\geq 0$$
  for all $s>0$. Since 
  $$\E\left( \int_0^t |\Aa_-f({\mathcal X}(r+s))|\dd r \right) = \int_s^{t+s}|\langle \pi_r, \Aa_- f\rangle | \dd r < \infty,$$
  the above process is actually a true martingale, and taking expectations, we arrive at the identity 
  $$\langle \pi_t, f\rangle- \langle \pi_s, f\rangle= \int_s^t \langle \pi_r, \Aa_-f\rangle  \dd r.$$
 
  With the observation above, we can let $s\to 0+$, and since $\langle \pi_s, f\rangle\to 0$ 
  thanks to \autoref{Lentrance}, we conclude that 
  $$\langle \pi_t, f\rangle= \int_0^t \langle \pi_r, \Aa_-f\rangle  \dd r.$$
\end{proof}

We conclude this section with the following corollary, which
demonstrates the existence of a solution to the growth-fragmentation equation
started from zero mass.
\begin{cor}
  Suppose that the hypotheses of \autoref{Centrance} are fulfilled.
  Let 
  \linebreak 
  $\gamma_t(\dd x)=x^{-\omega_{_-}}\pi_t(\dd x)$
  for $t>0$, and set $\gamma^{  }_0\equiv 0$. Then, the family $(\gamma^{  }_t)_{t\geq0}$ 
  solves \eqref{EqFG} for all $f\in{\mathcal C}^{\infty}_c(0,\infty)$
  when ${\mathcal L}={\mathcal L}_{\alpha}$ is given by \eqref{Eqop1}. 
\end{cor}
 
\section{Explosion of the stochastic model}
\label{s:explosion}

In this section, we discuss the behaviour of the stochastic growth-fragmentation
process in a simplified setting. 
In particular, we point out that,
when the Malthusian hypotheses from \autoref{s:ss} do not hold,
this stochastic model may experience explosion, in the sense that some arbitrarily small
compact sets contain infinitely many particles after a finite time.
We will focus on the case where $\alpha < 0$, though similar arguments
can be made for $\alpha > 0$.

\skippar
Assume that the measure $K$ is a probability measure on $[1/2,1)$
which is not equal to $\delta_{\frac{1}{2}}$, and denote by $Y$
a random variable with law $K$. Choose $c \in \R$ such that
$\E[\log(1-Y)] + c < 0 < \E[\log Y] + c$. 

We now set up the stochastic model. Let $\tree = \bigcup_{n \ge 0} \{L,R\}^n$,
where $\{L,R\}^0 = \{\varnothing\}$.
We view this as a binary tree, as follows. The root node $\eve$ gives rise
to child nodes $L$ and $R$; the former then has children $LL$ and $LR$, while
the latter has children $RL$ and $RR$, and so on. We introduce the
ancestry relationship `$\parenteq$' by saying that,
for individuals $u,u' \in \tree$, $u \parenteq u'$
if and only if there exists $u'' \in \tree$ such that $uu'' = u'$;
we also define the strict relation $u \parent u'$ to mean that
$u \parenteq u'$ but $u \ne u'$.

Let $\rays = \{L,R\}^\N$.
This is the set of infinite lines of descent,
or rays,
in $\tree$. For instance, $LLRRLRLRRRRL\dotsb \in \rays$
traces a line of descent starting at individual $\eve$,
and proceeding to $L$, then $LL$, then $LLR$, and so on.
If $u \in \tree$ and $v \in \rays$, we say (by slight
abuse of notation) that $u \parent v$ if and only if
there exists $v' \in \rays$ such that $uv' = v$.

To each $u \in \tree$, we assign, independently
of everything else,
a \emph{lifetime} $T_u$ which is
has an exponential distribution of rate 1,
and an \emph{offspring distribution} $Y_u$
which is distributed with law $K$. We then recursively
assign positions $\zeta_u$ to the individuals in $\tree$.
The root is positioned at a given point $x \in \R$, that is, 
$\zeta_{\eve} = x$.
Its descendents are positioned as follows:
\[ 
  \zeta_{uL} = \zeta_u + \log(1-Y_u) + cT_u \quad \text{and} \quad
  \zeta_{uR} = \zeta_u + \log(Y_u) + cT_u, \qquad u \in \tree.
\]
This gives a model in which each individual lives an exponential time,
dies, and (on average) scatters one child to the left and one to the right.

It will be convenient to introduce a model in which individuals also
move continuously, as follows.
For $u \in \tree$, define its birth time $a_u = \sum_{u'\parent u} T_{u'}$
and its death time $b_u = \sum_{u'\parenteq u} T_{u'} = a_u + T_u$. Its position
between those times is then given by $\xi_u(t) = \zeta_u + c(t-a_u)$,
for $a_u \le t < b_u$. By another abuse of notation, let us define also
the positions of a ray: for $v \in \rays$, let $\xi_v(t) = \xi_u(t)$, where
$u \in \tree$ is the unique individual with $a_u \le t < b_u$ and $u \parent v$.
We may now see the model as containing individuals which move to the right
at constant rate $c$, until an exponential clock rings and the individual
dies, scattering offspring to the left.

The model may also be viewed as a stochastic process 
$\Yy = (\Yy(t))_{t\ge 0}$, with
$$\Yy(t) = \sum_{u \in \tree} \delta_{\exp(\xi_u(t))} \Indic{a_u\le t< b_u},$$
taking values in the space $\Nn$ of locally finite point measures
and such that $\Yy(0) = \delta_{\e^x}$.
With this perspective, it is 
an important and useful fact that the process has the branching property.
Loosely speaking, this means that the behaviour of $(\Yy(t+s))_{s\ge 0}$
is given by collecting the atoms of $\Yy(t)$ and running from 
each one an independent copy of $\Yy$;
for a precise statement and proof see, for instance, \cite[Proposition 2]{BeCF}.

The process $\Yy$ we have just described is a stochastic model
corresponding to the homogeneous fragmentation equation. 
In particular, if we define a collection of measures $(\mu_t)_{t\ge 0}$ via
\[ \ip{\mu_t}{f} = \E\bigl[\ip{\Yy(t)}{f}\bigr] = \E\biggl[ \sum_{u \in \tree} f(\exp(\xi_u(t)))\Indic{a_u \le t < b_u} \biggr],
  \qquad f \in \Ctest,
\]
then we obtain a solution to \eqref{EqFG} with $\alpha=0$ and $\mathcal{L}$ given as in \eqref{Eqop0};
we give some more details on this in \autoref{s:bps}.
This corresponds to
the function $\kappa$ satisfying, for $q \ge 0$,
\[ \kappa(q) = cq + \int_{[\frac{1}{2},1)} \bigl[ y^q+(1-y)^q-1 \bigr] \, K(\dd y)
  > q\Big(c + \int_{[\frac{1}{2},1)} \log y \, K(\dd y) \Bigr) > 0.
\]
Here, the first inequality holds because it holds for the integrands, and
the second inequality is by our assumption about $c$ at the beginning
of the section. Also, $\kappa(0) = 1$.
Thus, $\inf_{q \ge 0} \kappa(q) > 0$, and so the Malthusian hypothesis \eqref{EqMalthus},
which was an important assumption in \autoref{s:ss}, is
not satisfied for our model.

We now give the model corresponding to the self-similar equation. To do this,
we should first introduce the notion of a \emph{stopping line}. We say
that $S = (S_v)_{v\in\rays}$ is a stopping line if:
\begin{enumerate}[(i)]
  \item For every $v \in \rays$, $S_v$ is a stopping time for the natural filtration of $\xi_v$;
  \item If $u \in \tree$ and $v, v' \in \rays$ such that $u \parent v,v'$, then $\P(S_v = S_{v'} \given a_u \le S_v < b_u) = 1$.
\end{enumerate}
Now, for $v \in \rays$ and $\alpha \in \R$, let
\[ T_v(s) =  \int_0^s \e^{-\alpha \xi_v(r)} \, \dd r , \]
and denote its inverse by $S_v$. Then, $S(t) = (S_v(t))_{v \in \rays}$ is 
a stopping line for every $t \ge 0$.
If $u \in \tree$ is an individual
such that, for some $v \in \rays$ with $u\parent v$, $a_u\le S_v(t)< b_u$ holds,
then we define $S_u(t)$ to be equal to $S_v(t)$;
by property (ii) of the definition of a stopping line,
this does not depend on the choice of $v$.
We define
\[ X_v(t) = \exp(\xi_v(S_v(t))), \qquad v \in \rays, \; t \ge 0 , \]
%
and
\begin{equation}\label{e:def-Xf}
  \Xf(t) = \sum_{u \in \tree} \delta_{\exp \xi_u(S_u(t))} \Indic{a_u \le S_u(t) < b_u}, \qquad t \ge 0,
\end{equation}
where the sum is over only those $u$ for which $S_u(t)$ is defined.
The process $\Xf$ is called the
$\alpha$-self-similar fragmentation process. 
The stopping-line nature of $S$ means that the process $\Xf$ retains
the branching property; however, it is not clear that it should be locally finite,
and indeed,
our main result in this section is that $\Xf$ a.s.\ does \emph{not}
remain locally finite for all time:

\begin{prop}
  Let $\alpha < 0$. For any $a > 0$, there exists some random time $\sigma$ such that
  there are infinitely many individuals of $\Xf(\sigma)$ in the compact set $[1,1+a]$.
\end{prop}
\begin{proof}
  We will study rays $p_k = L^kR^\infty$ which follow the left-hand offspring
  for $k$ steps, and the right-hand offspring thereafter. Our first remark
  is that, if we define
  \[ \tau_0(v) = \inf\{ t\ge 0 : X_v(t) = 0\} , \qquad v \in \rays, \]
  then we have
  \[
    \tau_0 \coloneqq \tau_0(L^\infty)
    = T_{L^\infty}(\infty) = \int_0^\infty \e^{-\alpha \xi_{L^{\infty}}(t)}\, \dd t.
  \]
  Since $\xi_{L^\infty}$ is a Lévy process with negative mean (by our assumption on $c$)
  we obtain that $\tau_0 < \infty$ almost surely.

  Consequently, for any $\eta > 0$ there exists some infinite set $C$ such that,
  for $k \in C$,
  $T_{p_k}(b_{L^k}) \in (\tau_0-\eta, \tau_0)$ and $X_{p_k}(T_{p_k}(b_{L^k})) \to 0$ as $k \to \infty$.
  We define 
  \[ L_1^+(v) = \sup\{ t\ge 0 : X_v(t) \le 1 \} , \qquad v \in \rays, \]
  which is the \emph{last passage time} of the level 1 by the process $X_v$;
  then, for $k \in C$, we have
  \[ L_1^+(p_k) = T_{p_k}(b_{L^k}) + \tilde L_1^+(R^\infty) , \]
  where $\tilde L_1^+(R^\infty)$ is the last passage time of the level 1
  by $\tilde X_{R^\infty}$, computed for an
  independent self-similar fragmentation process started at $X_{p_k}(b_{L^k})$.

  We are therefore reduced to studying first passage times of the
  $(-\alpha)$-pssMp $X_{R^\infty}$ corresponding to a spectrally negative Lévy process $\xi_{R^\infty}$
  started from a level $x < 0$ with Laplace
  exponent
  \[ \Phi_{R^\infty}(q) = cq + \int_{[\frac{1}{2},1)} (y^q-1) K(\dd y). \]
  We seek  $t$ such that, with positive probability (not depending
  on $x$), $\tilde L_1^+({R^\infty}) \le t$.

  The first observation is that $\Phi'_{R^{\infty}}(0+)>0$, which implies that the process $X_{R^\infty}$
  can be extended to start at zero; it is then Feller on $[0,\infty)$. Furthermore,
  the pssMp drifts to $+\infty$ as $t\to \infty$.
  Hence, we pick $\epsilon>0$ and $t \ge 0$ such that $\P_0(L_1^+ \le t) \ge 2\epsilon$.
  By the Portmanteau theorem, $\liminf_{x \to 0}\P_{x}(L_1^+ \le t) \ge 2\epsilon$ also;
  and therefore for $x$ sufficiently close to zero, $\P_x(L_1^+ \le t) \ge \epsilon$.
  Applying the Borel-Cantelli lemma to the paths referred to above,
  there are then infinitely many $k \in C$ such that
  \[ L_1^+(p_k) \le T_{p_k}(b_{L^k}) + t \le \tau_0 + t \]
  with probability 1.

  We therefore have infinitely many paths whose last passage time of $1$
  occurs in a (random) compact interval. In particular, there must exist
  some finite random time $\sigma$ such that for every $\delta> 0$, 
  there are infinitely many paths which cross $1$ for the last
  time in $(\sigma-\delta,\sigma)$. 
  Since the particles are large, the time-change is bounded,
  in that $S_{p_k}(L^+_1(p_k)+u)-S_{p_k}(L_1^+(p_k)) \le u$ for all $u\ge 0$.
  Furthermore, the processes $\xi_{p_k}$
  can grow at most at rate $c$;
  thus, picking $\delta < c^{-1}\log(1+a)$
  ensures that,
  at time $\sigma$, all the selected particles are in $[1,1+a]$,
  which completes the proof.
\end{proof}

\skippar
This result illustrates one example where the Malthusian hypothesis, under we
examined the growth-fragmentation equation, fails, and where the stochastic model $\Xf$
(whose mean measure
could otherwise be expected to give a solution)
does not remain locally finite. However, since
uniqueness generally fails for the self-similar growth-fragmentation equation, it
does not immediately imply that there is no global solution of \eqref{EqFG}.

The procedure
of creating the growth-fragmentation process $\Yy$ can be carried out under
the general conditions on $a$, $b$ and $K$ given in the main body of the paper
(this is done in \citet{BeCF}) and the self-similar
time-change \eqref{e:def-Xf} can also be applied in this general context for any $\alpha \ne 0$
(see \cite[Corollary 2]{BeMGF}.) However,
it remains an open problem to determine necessary and sufficient conditions for
the process $\Xf$ to be locally finite at all times, and to decide when global solutions
of \eqref{EqFG} exist.

\section{Branching particle system and many-to-one formulas}
\label{s:bps}

In this concluding section, we aim to clarify the connection between
our work and the
`spine' or `tagged fragment' approach to branching particle systems
and fragmentation processes, with the hope of explaining the source
of the solutions obtained in
\autoref{T1} and Corollaries \ref{c:mu-alpha-neg} and \ref{c:mu-alpha-pos}.
We refer the reader to the survey of \citet{HH-spine} for results in the context
of branching processes, \citet[\S 3.2.2]{Ber-frag} for the background on
fragmentation processes, and \citet{Haas} for the use of the tagged fragment
in solving the classical fragmentation equation.
For the sake of simplicity, we focus on the case when
the dislocation measure $K$ is finite
and the operator ${\mathcal L}$ has the form \eqref{Eqop0}.


\medskip\noindent
Let us assume that we can construct a system of branching particles in $(0,\infty)$, 
with the following dynamics: particles evolve independently of one other, 
each particle located at $x>0$ grows at rate $cx^{\alpha+1}$, 
and a particle located at $x>0$ is replaced by 
two particles located respectively at $xy$ and $x(1-y)$ at rate $x^{\alpha} K(\d y)$. 
Let $\mathbf{Z}(t)=(Z_i(t))_{i\ge 1}$ denote the collection of particles in the system at time $t\geq 0$,
starting at time $0$ from a single particle located at $1$.
Informally, the verbal description of the dynamics of the particle system suggests 
that for every test function $f\in{\mathcal C}^{\infty}_c(0,\infty)$,
the functional $F(\mathbf{z})=\sum_i f(z_i)$ for $\mathbf{z}=(z_i)_{i\ge 1}$
belongs to the domain of the infinitesimal generator ${\mathcal G}$ 
of the process $(\mathbf{Z}(t))_{t\geq 0}$, and that
\[
 \mathcal{G}F(\mathbf{z})=\sum_{i\ge 1} z_i^{\alpha} \biggl( c z_i f'(z_i)+\int_{[1/2,1)}(f(y\mid z_i)-f(z_i))K(\d y)\biggr).
\]
Therefore, if we write $\mu_t$ for the intensity measure of $\mathbf{Z}(t)$, i.e.
\[ \langle \mu_t,f\rangle = \E(F(\mathbf{Z}(t)) = \E\biggl(\sum_i f(Z_i(t))\biggr), \]
then Kolmogorov's forward equation entails that $(\mu_t)_{t\geq 0}$ solves the fragmentation equation \eqref{EqFG}.

The analysis of the system $(\mathbf{Z}(t))_{t\ge 0}$ can be significantly
simplified by identifying a \emph{spine} among the particles, and formulating
questions about the entire system in terms of just the spine particle via a \emph{many-to-one}
formula. This proceeds roughly as follows. Suppose that we can identify
a function $(t,z) \mapsto \varphi(t,z) > 0$ such that
the process $M(t) = \sum_{i \ge 1} \varphi(t,Z_i(t))$ is a martingale. We introduce
a new probability measure $\tilde \P$ by means of a martingale change of measure using $M$,
simultaneously identifying one of the particles to be the spine; specifically, we
identify particle $i$ of $\Zb(t)$ as the spine with $\tilde \P$-probability
proportional to $\varphi(t,Z_i(t))$. 

At each time $0 \le s\le t$, the spine particle has a unique ancestor,
and we define the random variable $W(s)$
to be the position of this ancestor.
We now aim to identify the law of certain functionals of $\Zb$
in terms of the law of $W$. In particular, if we define
$\rho_t$ to be the law of $W(t)$, we obtain
\[
  \ip{\mu_t}{f} = \E\biggl( \sum_{i\ge 1} f(Z_i(t)) \biggr)
  = \tilde \E\biggl(\frac{f(W(t))}{\varphi(t,W(t))} \biggr)
  = \ip{\rho_t}{f/\varphi(t,\cdot)},
\]
which is known as a many-to-one formula.

The spine method for solving \eqref{EqFG}
can be summarised as follows. We first use the dynamics of the branching particle system
and the effect of the martingale change of measure to identify the process $W$.
The one-dimensional distributions of $W$ give the collection of measures 
$(\rho_t)_{t \ge 0}$, and then the many-to-one formula
gives us an explicit description of $(\mu_t)_{t \ge 0}$.

\medskip\noindent
The method we have sketched can be made rigorous in the homogeneous case $\alpha=0$,
even in the more general situation when the dislocation measure $K$ is infinite and fulfills \eqref{Eqcondnu}.
More precisely,
one can take $\varphi(t,z)=\exp(-\kappa(\omega)t)z^{\omega}$ 
for any $\omega\in \dom \kappa$, 
and then the process $W(t)$ is the exponential of a L\'evy process
with no positive jumps and Laplace exponent $\Phi_{\omega} = \kappa(\omega+\cdot)-\kappa(\omega)$;
this is the justification for \autoref{r:LE}.
We stress, however, that the general self-similar case $\alpha\neq 0$ is far less simple.
In particular, it is not clear whether the branching particle system can indeed
be constructed since, as noted in the previous section, explosion may occur.
A fairly general class of growth-fragmentation processes was introduced recently 
in \cite{BeMGF} by means of a Crump-Mode-Jagers process;
however, although it is expected to be related to growth-fragmentation equations as described above,
so far no many-to-one formula is known to make the connection rigorous.

\acks

This work was submitted while the second author was at the University of Zurich, Switzerland.
We thank Robin Stephenson for drawing to our attention some mistakes in an earlier draft,
and the anonymous referee for their helpful comments.



\end{document}